\documentclass[a4paper,10pt]{amsart}

\usepackage[english]{babel}
\usepackage[utf8]{inputenc}

\usepackage{amssymb}
\usepackage{amsmath}
\usepackage{amsthm}
\usepackage{mathtools}
\usepackage{bbm}

\usepackage{enumitem}
\usepackage{tikz}
\usepackage{marginnote}

\newcommand{\bbC}{\mathbb{C}}
\newcommand{\bbD}{\mathbb{D}}

\newcommand{\bbN}{\mathbb{N}}

\newcommand{\bbR}{\mathbb{R}}

\newcommand{\calK}{\mathcal{K}}
\newcommand{\calL}{\mathcal{L}}

\newcommand{\suchthat}{\,|\,} 
\DeclareMathOperator{\one}{\mathbbm{1}} 
\newcommand{\argument}{\mathord{\,\cdot\,}} 
\newcommand{\dx}{\;\mathrm{d}} 
\DeclareMathOperator{\conv}{conv} 
\newcommand{\norm}[1]{\left\lVert #1 \right\rVert} 
\newcommand{\modulus}[1]{\left\lvert #1 \right\rvert} 
\newcommand{\inner}[2]{\left(#1 \mid #2\right)}

\newcommand{\impliesProof}[2]{``\ref{#1} $\Rightarrow$ \ref{#2}''}

% ...

\theoremstyle{definition}
\newtheorem{definition}{Definition}[section]
\newtheorem{remark}[definition]{Remark}
\newtheorem{remarks}[definition]{Remarks}

\newtheorem{example}[definition]{Example}
\newtheorem{examples}[definition]{Examples}
\newtheorem{open_problem}[definition]{Open Problem}

\theoremstyle{plain}
\newtheorem{proposition}[definition]{Proposition}
\newtheorem{lemma}[definition]{Lemma}
\newtheorem{theorem}[definition]{Theorem}
\newtheorem{corollary}[definition]{Corollary}

\numberwithin{equation}{section}

\begin{document}

\title[Increasing sequences in ordered Banach spaces]{Increasing sequences in ordered Banach spaces -- new theorems and open problems}
\author{Jochen Gl\"uck}
\address[J. Gl\"uck]{Bergische Universit\"at Wuppertal, Fakult\"at f\"ur Mathematik und Naturwissenschaften, Gaußstr.\ 20, 42119 Wuppertal, Germany}
\email{glueck@uni-wuppertal.de}
\date{\today}
\begin{abstract}
	An ordered Banach space $X$ is said to have the Levi property or to be regular if every increasing order bounded net (equivalently, sequence) is norm convergent. We prove four theorems related to this classical concept: 
	
	(i) The Levi property follows from the -- formally weaker -- assumption that every increasing net that has a minimal upper bound is norm convergent. This motivates a discussion about in which sense the Levi property resembles the notion of order continuous norm from Banach lattice theory. 
	
	(ii) If $X$ is separable and has normal cone, then the assumption that every increasing order bounded sequence has a supremum implies the Levi property. This generalizes a classical result about Banach lattices, but requires new ideas since one cannot work with disjoint sequences in the proof.
	
	(iii) A version of Dini's theorem for ordered Banach spaces that is more general than what is typically stated in the literature. We use this to derive a sufficient condition for the space of all compact operators between two Banach lattices to have the Levi property.
	
	(iv) Dini's theorem never holds on reflexive ordered Banach spaces with non-normal cone -- i.e., on such a space one can always find an increasing sequence that converges weakly but not in norm.
	
	We illustrate our results by various examples and counterexamples and pose four open problems. 
\end{abstract}

\maketitle

\section{Introduction} 
\label{sec:introduction}

\subsection*{Ordered Banach spaces}

By a \emph{pre-ordered Banach space} we mean a real Banach space $X$ together with a closed \emph{wedge} $X_+$ in $X$, 
i.e.\ a closed non-empty subset $X_+ \subseteq X$ that satisfies $\alpha X_+ + \beta X_+ \subseteq X_+$ 
for all scalars $\alpha, \beta \in [0,\infty)$. 
The wedge $X_+$, sometimes also called the \emph{positive wedge} in $X$, 
induces a pre-order (i.e.\ a reflexive and transitive relation) $\le$ on $X$ 
that is compatible with the vector space structure and that is defined by $x \le y$ if and only if $y-x \in X_+$. 
Hence, the elements of $X_+$ are precisely the vectors $x \in X$ that satisfy $x \ge 0$ 
and we thus call them the \emph{positive} vectors in $X$.
The wedge $X_+$ is called \emph{pointed} or a \emph{cone} if it satisfies $X_+ \cap -X_+ = \{0\}$.  
The pre-order $\le$ is antisymmetric (and thus a partial order) if and only if $X_+$ is a cone. 
In this case, $X$ is called an \emph{ordered Banach space}. 
Throughout the paper we will mostly be interested in ordered Banach spaces, 
but for some situations -- in particular regarding duality -- it is more suitable 
to have the concept of pre-ordered Banach spaces available.

The importance of ordered Banach spaces stems from the fact that many Banach spaces that occur throughout mathematics 
and in various applied fields carry a natural order structure. 
This is, in particular, true for numerous function spaces and also for various non-commutative spaces 
such as the self-adjoint parts of $C^*$-algebras or non-commutative $L^1$-spaces. 

The major themes in the study of ordered Banach spaces are, on the one hand, 
the structure of the spaces themselves and, on the other hand, the analysis of positive operators between them.
Recent research topics about the structure of ordered Banach spaces include 
the study of non-lattice ordered spaces by embedding them into vector lattices \cite{KalauchMalinowski2019a, KalauchStennderVanGanns2021, KalauchVanGaans2019},
the order structure of Sobolev spaces \cite{AroraGlueckSchwenningerPreprint, PonceSpector2020}, 
and topological and metric properties of the decomposition of vectors 
as the difference of two positive vectors \cite{DeJeuMesserschmidt2014, Messerschmidt2019}.
Research about positive operators include results about the spectral theory and the long term behaviour 
of positive operators and semigroups \cite{GlueckWeber2020, LiJia2021, FonteSanchezGabrielMischlerPreprint}, 
the analysis of positive operators to describe quantum systems \cite{FidaleoOttomanoRossi2022, FidaleoVincenzi2023, HuangJaffeLiuWuPreprint, RajaramaBhatKarTalwar2023}, 
structure theorems for spaces of positive operators \cite{vanGaansKalauchRoelands2024}, 
as well as positive infinite-dimensional systems \cite{AroraGlueckPaunonenSchwenningerPreprint, ElGantouh2023, ElGantouh2024, ElGantouhPreprint} and positive perturbations \cite{BarbieriEngelPreprint, BatkaiJacobVoigtWintermayr2018}. 
The later two topics lead, in turn, back to questions 
about geometric properties of ordered Banachs spaces \cite{AroraGlueckPaunonenSchwenningerPreprint, AroraGlueckSchwenningerPreprint}.

\subsection*{Contributions} 

In this article we are interested in the norm convergence of increasing (or decreasing) sequences and nets in ordered Banach spaces; 
see the end of the introduction for a precise definition of those notions. 
In many examples of ordered Banach spaces every increasing net (equivalently sequence, see Theorem~\ref{thm:oc-norm}) 
that is order bounded from above converges in norm. 
This is called the \emph{Levi property} in parts of the literature; 
see Remark~\ref{rems:oc-norm}\ref{rems:oc-norm:itm:name} for references to this and various other names for the same property. 
We prove four results related to the Levi property. 

First, we show that the Levi property is equivalent to the formally weaker property 
that every increasing net which has a minimal upper bound 
converges in norm (Theorem~\ref{thm:oc-norm}). 
This gives rise to a comparison with the notion of \emph{order continuous norm} 
in Banach lattices, see in particular Remark~\ref{rems:oc-norm}\ref{rems:oc-norm:itm:oc-norm-discussion}.

Second, we prove that if an ordered Banach space with normal cone is separable (in the norm topology) 
and has the property that every order bounded increasing sequence has a supremum, 
then the space automatically has the Levi property (Theorem~\ref{thm:separable}). 
For the special case of Banach lattices, this is a classical and well-known result. 
The main challenge in the much more general setting of ordered Banach spaces is that one cannot argue with disjoint sequences. 

Third, we show a version of Dini's theorem in ordered Banach spaces with normal cone, 
which yields that, for increasing nets, a rather general version of weak convergence automatically implies norm convergence 
(Theorem~\ref{thm:dini}). 
The argument itself is classical, but our formulation of the result has the advantage that it can be used 
to give a sufficient condition for the space of all compact operators between two Banach lattices 
to have the Levi property (Corollary~\ref{cor:compact-operators-oc-norm}). 

Fourth, we prove that for ordered Banach spaces that are reflexive, 
the validity of Dini's theorem implies that the cone is normal (Theorem~\ref{thm:no-dini}).

Section~\ref{sec:bidual-wedge-pointed} contains a brief intermezzo regarding the question 
under which conditions the bidual wedge of an ordered Banach $X$ is pointed, i.e.\ a cone. 
This property is not trivial if the cone $X_+$ is not normal; we characterize 
the property in Proposition~\ref{prop:bidual-cone-dom}.
In Appendix~\ref{app:weak-star-dense} we prove by means of the bipolar theorem from the theory of locally convex spaces 
that the positive unit ball of a pre-ordered Banach space $X$
is weak* dense in the positive unit ball of the bidual space $X''$ (Theorem~\ref{thm:positive-ball-bidual}). 
This result is employed in the proof of Proposition~\ref{prop:bidual-cone-dom}. 

We illustrate our results with various examples and include four open problems;
those can be found in Open Problems~\ref{op:oc-norm}, \ref{op:sup-implies-normal},~\ref{op:dini-normal} and~\ref{op:dini-sequences}.

\subsection*{Notation and terminology} 

We use the convention $\bbN := \{1, 2, \dots\}$ and $\bbN_0 := \{0,1,2, \dots\}$.

The wedge $X_+$ in a pre-ordered Banach space $X$ is called \emph{generating} if its linear span $X_+ - X_+$ equals $X$; 
it is called \emph{total} if its linear span is dense in $X$. 
If $X_+$ is generating, then there exists a number $M \ge 0$ with the following property: 
every $x \in X$ can be written as $x = y-z$ for two vectors $y,z \in X_+$ 
that satisfy $\norm{y}, \norm{z} \le M \norm{x}$; see e.g.\ \cite[Proposition~1.1.2]{BattyRobinson1984} 
(note that in this reference, the notion \emph{cone} is used differently than in our paper: 
they call a cone what we call a wedge). 
The wedge $X_+$ is called \emph{normal} if there exists a constant $C \ge 0$ such that one has $\norm{x} \le C \norm{y}$ 
for all $x,y \in X$ that satisfy $0 \le x \le y$. 
Note that a normal wedge is automatically a cone.
For $x,z \in X$ the set $[x,z] := \{y \in X \suchthat x \le y \le z\}$ is called the \emph{order interval} 
between $x$ and $z$. 
Obviously, $[x,z]$ is non-empty if and only if $x \le z$. 
The wedge $X_+$ is normal if and only if every order interval is norm bounded 
(see e.g.\ \cite[Theorem~2.38(5)]{AliprantisTourky2007} for the case where $X_+$ is a cone; 
if $X_+$ is not a cone, than the order interval $[0]$ is a non-zero vector subspace of $X$ 
and is thus not norm bounded).

Now, let $X$ be an ordered Banach space.
A set $S \subseteq X$ is called \emph{order bounded} if it is contained in an order interval. 
It is called \emph{order bounded from above} 
if there exists a vector $z \in X$ such that $y \le z$ for all $y \in S$; 
in this case, any such vector $z \in X$ is called an \emph{upper bound} of $S$. 
An upper bound $z$ of $S$ is called a \emph{minimal} upper bound of $S$ 
if for every upper bounded $\tilde z$ of $S$ the inequality $\tilde z \le z$ implies $\tilde z = z$.
An upper bounded $z$ of $S$ is called a \emph{smallest upper bound} or \emph{supremum} of $S$ 
if every upper bound $\tilde z$ of $S$ satisfies $\tilde z \ge z$.
If a supremum of $S$ exists, then it is uniquely determined; 
moreover, the supremum of $S$ is then the only minimal upper bound of $S$.
On the other hand, a set $S$ that does not have a supremum can have more than one minimal upper bound. 
The notions \emph{order bounded from below}, \emph{lower bound}, 
\emph{maximal lower bound}, and \emph{greatest lower bound} (or \emph{infimum}) 
are defined analogously. 
Note that notions such as minimal upper bounds and suprema are not useful concepts in pre-ordered 
(rather than ordered) Banach spaces.

A linear map $T: X \to Y$ between two pre-ordered Banach spaces is called \emph{positive}, 
which we denote by $T \ge 0$, if $TX_+ \subseteq Y_+$. 
If the wedge $X_+$ in $X$ is generating and the wedge $Y_+$ is a cone, 
then every positive linear map $X \to Y$ is automatically continuous 
(see e.g.\ \cite[Theorem~2.32]{AliprantisTourky2007} or \cite[Theorem~2.8]{ArendtNittka2009}). 
We denote the space of bounded linear operators between two Banach spaces $X$ and $Y$ by $\calL(X;Y)$; 
its subspace of all compact operators is denoted by $\calK(X;Y)$. 
We write $X'$ for the norm dual of a Banach space $X$. 
For $x' \in X'$ and $x \in X$ we set $\langle x', x \rangle := x'(x)$; 
when it seems appropriate to explicitly indicate the spaces $X$ and $X'$ 
we sometimes write $\langle x', x\rangle_{\langle X', X \rangle}$ instead.

Let $X$ be a pre-ordered Banach space. 
It follows from the preceding definition of positive operators that a functional $x' \in X' = \calL(X; \bbR)$ 
is positive if and only if it maps $X_+$ into $[0,\infty)$. 
We denote the set of all positive elements of $X'$ by $X'_+$. 
The set $X'_+$ is a weak* closed wedge in $X'$. 
In particular, it is norm closed in $X'$, so it turns $X'$ into a pre-ordered Banach space.
The dual wedge $X'_+$ is a cone 
if and only if the wedge $X_+$ in $X$ is total. 
It follows from the Hahn--Banach separation theorem that a vector $x \in X$ 
is contained in $X_+$ if and only if $\langle x', x \rangle \ge 0$ for all $x' \in X'_+$. 
A functional $x' \in X'$ is called \emph{strictly positive} 
if $\langle x', x \rangle > 0$ for every $0 \not= x \in X_+$.
By taking the dual wedge of the dual wedge one obtains the \emph{bidual wedge} $X''_+$ in the bidual space $X''$.
The wedge $X_+$ is normal if and only if the dual wedge is generating, i.e.\ $X'_+ - X'_+ = X'$ 
\cite[Theorem~4.5]{KrasnoselskijLifshitsSobolev1989}. 
Similarly, the wedge $X_+$ is generating if and only if $X'_+$ is normal  
\cite[Theorem~4.6]{KrasnoselskijLifshitsSobolev1989}.

Let $X$ be an ordered Banach space again.
A net $(x_j)_{j \in J}$ is called \emph{increasing} 
if $x_j \le x_k$ for all $j,k \in J$ that satisfy $j \le k$. 
Similarly, the net is called \emph{decreasing} if $x_j \ge x_k$ whenever $j \le k$. 
For a net $(x_j)_{j \in J}$ in $X$, when we speak of notions such as 
\emph{order bounded from above}, \emph{upper bound}, \emph{minimal upper bound}, or \emph{supremum}, 
we actually mean the same notions for the set $\{x_j \suchthat j \in J\}$.
If an increasing net $(x_j)_{j \in J}$ converges weakly (or even in norm) 
to a point $x \in X$, then it is not difficult to check that $x$ is the supremum of the net.

\section{Norm convergence of increasing sequences and nets} 
\label{sec:increasing}

A Banach lattice $X$ is said to have order continuous norm if every increasing net in $X$ 
which has a supremum converges in norm (to the supremum). 
Somewhat surprisingly at first glance, this is equivalent to the formally much stronger property 
that every increasing net in $X$ that is order bounded from above converges in norm. 
Formally, the second property is equivalent to the first one plus Dedekind completeness, 
but remarkably it turns out that the Dedekind completeness comes for free from the first property. 
In the following theorem we consider the situation in the more general setting of ordered Banach spaces, 
which turns out to be subtler. 
The main difference is that, in~\ref{thm:oc-norm:itm:net-with-min-upper-bound}, we require norm convergence 
of every increasing net which has a minimal upper bound rather than only for those increasing nets which have a smallest upper bound 
(i.e., a supremum). 

\begin{theorem}
	\label{thm:oc-norm}
	Let $X$ be an ordered Banach space. 
	The following are equivalent:
	\begin{enumerate}[label=\upshape(\roman*)]
		\item\label{thm:oc-norm:itm:sequence} 
		Every increasing sequence in $X$ that is order bounded from above is norm convergent, 
		i.e.\ $X$ has the \emph{Levi property}.
		
		\item\label{thm:oc-norm:itm:net} 
		Every increasing net in $X$ that is order bounded from above is norm convergent. 
		
		\item\label{thm:oc-norm:itm:net-with-min-upper-bound} 
		Every increasing net in $X$ that has a minimal upper bound is norm convergent.
	\end{enumerate}
	If the equivalent conditions~\ref{thm:oc-norm:itm:sequence}--\ref{thm:oc-norm:itm:net-with-min-upper-bound} are satisfied, 
	then the cone $X_+$ is normal.
\end{theorem}

Before we give the proof, a number of remarks are in order. 

\begin{remarks}
	\label{rems:oc-norm}
	\begin{enumerate}[label=(\alph*)]
		\item\label{rems:oc-norm:itm:positive-only} 
		In assertions~\ref{thm:oc-norm:itm:sequence} and~\ref{thm:oc-norm:itm:net} of the theorem, 
		it suffices to consider sequences and nets in $X_+$, respectively. 
		To see this for~\ref{thm:oc-norm:itm:sequence} one can simply substract the first element of the sequence from the entire sequence. 
		To use the same argument for~\ref{thm:oc-norm:itm:net} one first has to switch to a tail of the net.
		
		\item\label{rems:oc-norm:itm:name} 
		As indicated in the theorem, assertion~\ref{thm:oc-norm:itm:sequence} is called the \emph{Levi property}  
		in \cite[Definition~2.44(2) on p.\,89]{AliprantisTourky2007}.
		In \cite[Definition~5.1]{KrasnoselskijLifshitsSobolev1989} 
		the terminology \emph{regular cone} is used for this property. 
		In \cite[Definition~3.4]{ArendtDanersPreprint} the property is called \emph{order continuous norm}; 
		see part~\ref{rems:oc-norm:itm:oc-norm-discussion} of this remark for further discussion of this terminology.
		
		\item\label{rems:oc-norm:itm:lattice-case} 
		Assume that $X$ is a Banach lattice. 
		Then minimal upper bounds of sets are automatically smallest upper bounds, 
		so assertion~\ref{thm:oc-norm:itm:net-with-min-upper-bound} can be reformulated by saying that 
		every increasing net in $X$ that has supremum is norm convergent (to its supremum).
		This property is typically referred to as \emph{order continuity of the norm} of $X$ in Banach lattice theory, 
		see e.g.\ \cite[Definition~2.4.1]{MeyerNieberg1991}
		The equivalence of all three assertions in the theorem is a classical result for Banach lattices, 
		see for instance \cite[Theorem~2.4.2]{MeyerNieberg1991} or \cite[Theorem~II.5.10]{Schaefer1974}.
				
		\item\label{rems:oc-norm:itm:sequence-do-not-suffice} 
		The assertions of Theorem~\ref{thm:oc-norm} are not equivalent 
		to the weaker property that every increasing sequence that has a minimal upper bound is norm convergent 
		-- even for Banach lattices (where a minimial upper bound is the same as a supremum) 
		a simple counterexample can be found in \cite[Example~6 on p.\,46]{Wnuk1999}.
		
		\item\label{rems:oc-norm:itm:oc-norm-discussion} 
		We do not know whether, for general ordered Banach spaces, assertion~\ref{thm:oc-norm:itm:net-with-min-upper-bound} in the theorem 
		can be replaced with the formally weaker property 
		that every increasing net that has a supremum is norm convergent; 
		see Open Problem~\ref{op:oc-norm}. 
		If this is true, this would give a strong semantic justification to also call the equivalent properties of the theorem 
		\emph{order continuous norm}, as in the Banach lattice case. 
		If it is not true though, it appears to be not completely clear 
		how suitable the terminology \emph{order continuous norm} 
		actually is for the equivalent conditions in Theorem~\ref{thm:oc-norm}.
	\end{enumerate}
\end{remarks}
		
\begin{proof}[Proof of Theorem~\ref{thm:oc-norm}]
	It is known that~\ref{thm:oc-norm:itm:sequence} implies that the cone $X_+$ is normal, 
	see \cite[Theorem~2.45(2)$\Rightarrow$(3) on pp.\,89--90]{AliprantisTourky2007}. 
	We now show the equivalence of~\ref{thm:oc-norm:itm:sequence}--\ref{thm:oc-norm:itm:net-with-min-upper-bound} 
	(without making use of normality of the cone).
	
	\impliesProof{thm:oc-norm:itm:sequence}{thm:oc-norm:itm:net}
	Assume that~\ref{thm:oc-norm:itm:net} fails, 
	i.e.\ that there exists an increasing net $(x_j)_{j \in J}$ in $X$ that is order bounded from above but not norm convergent. 
	Since $X$ is norm complete this implies that $(x_j)_{j \in J}$ is not Cauchy. 
	Hence, we can find a number $\varepsilon > 0$ such that for every $j_0 \in J$ there exists $j_1 \ge j_0$ 
	which satisfies $\norm{x_{j_1}-x_{j_0}} \ge \varepsilon$. 
	Due to this property one can find an increasing sequence $(y_n)_{n \in \bbN}$ which consists of vectors from the net $(x_j)_{j \in J}$ 
	(but which might not be a subnet thereof) and which satisfies $\norm{y_{n+1} - y_n} \ge \varepsilon$ for each $n$. 
	Hence the sequence is not Cauchy and therefore not norm convergent, 
	which shows that~\ref{thm:oc-norm:itm:sequence} fails. 
		
	\impliesProof{thm:oc-norm:itm:net}{thm:oc-norm:itm:sequence} 
	This implication is obvious.
	
	\impliesProof{thm:oc-norm:itm:net}{thm:oc-norm:itm:net-with-min-upper-bound}
	This implication is obvious.
	
	\impliesProof{thm:oc-norm:itm:net-with-min-upper-bound}{thm:oc-norm:itm:net} 
	Assume that~\ref{thm:oc-norm:itm:net-with-min-upper-bound} holds 
	and let $(x_j)_{j \in J}$ be an increasing net that is order bounded from above. 
	Let $\emptyset \not= U \subseteq X$ denote the set of all upper bounds of $(x_j)_{j \in J}$ 
	and let $V$ be maximal with respect to inclusion among all downwards directed sets in $U$. 
	Such a set $V$ exists by Zorn's lemma. 
	The set $J \times V$ becomes a directed set if we endow $V$ with the direction opposite to its order inherited from $X$ 
	and $V \times J$ with the product direction. 
	The net $(v - x_j)_{(j,v) \in V \times J}$ is decreasing and bounded below by $0$. 
	
	Let us show that $0$ is a maximal lower bound of this net. 
	To this end, let $\ell \in X$ be a lower bound of $(v - x_j)_{(j,v) \in V \times J}$ and assume that $\ell \ge 0$. 
	Then $v - \ell \ge x_j$ for each $v \in V$ and each $j \in J$, 
	so $v - \ell \in U$ for each $v \in V$. 
	The set $V \cup (V-\ell)$ is downwards directed since $\ell \ge 0$ and is contained in $U$, 
	so we conclude from the maximality of $V$ that $V \cup (V-\ell) = V$ and hence, 
	$V-\ell \subseteq V$.  
	Now fix an arbitrary point $v \in V$. 
	Then it follows iteratively that $v - n\ell \in V \subseteq U$ for each $n \in \bbN$. 
	For an arbitrary index $j \in J$ we thus have $v - x_j \ge n \ell$ for each $n \in \bbN$. 
	As the cone $X_+$ is closed we conclude that $0 \ge \ell$ and hence $\ell = 0$. 
	So $0$ is indeed a maximal lower bound of $(v - x_j)_{(j,v) \in V \times J}$. 
	
	It thus follows from~\ref{thm:oc-norm:itm:net-with-min-upper-bound} that $(v - x_j)_{(j,v) \in V \times J}$ 
	is norm convergent and clearly, the limit is the infimum of the net and thus coincides with the maximal lower bound $0$. 
	We now show that $(x_j)_{j \in J}$ is Cauchy and hence norm convergent, as $X$ is a Banach space. 
	Let $\varepsilon > 0$. 
	As we have just shown, there exists an $j_0 \in J$ and a $v_0 \in V$ such that $\norm{v-x_j} \le \varepsilon$ 
	for all $j \ge j_0$ and all $v \le v_0$ (where the latter inequality refers to the order in $X$). 
	For each $j \ge j_0$ this implies that
	\begin{align*}
		\norm{x_j - x_{j_0}} 
		\le 
		\norm{x_j - v_0} + \norm{v_0 - x_{j_0}} 
		\le 
		2\varepsilon,
	\end{align*}
	so $(x_j)_{j \in J}$ is indeed a Cauchy net, as claimed.
\end{proof}

\begin{examples}
	\label{exas:almost-oc-norm}
	Each of the following ordered Banach spaces $X$ satisfies the equivalent conditions 
	of Theorem~\ref{thm:oc-norm}:
	\begin{enumerate}[label=(\alph*)]
		\item 
		Every reflexive ordered Banach space with normal cone. 
		Indeed, in a reflexive ordered Banach space every increasing order bounded sequence 
		is weakly convergent. 
		If the cone is normal, a Dini type result (see Corollary~\ref{cor:dini-weak-conv} below) implies 
		that a weakly convergent increasing sequence is automatically norm convergent.
		
		\item 
		More generally than in~(a), all ordered Banach spaces in which every order interval is weakly compact 
		(again, this follows from the version of Dini's theorem in Corollary~\ref{cor:dini-weak-conv} below).
		This class includes every $L^1$-space and, more generally, 
		(the self-adjoint part of) every pre-dual of a von Neumann algebra, 
		see e.g.\ \cite[Theorem~III.5.4(i) and~(iii)]{Takesaki1979}. 
		
		\item 
		As a special case of the previous point, all ordered Banach spaces in which every order interval is compact. 
		For a characterization of such spaces, see \cite{ChaudharyAtkinson1978}.
		This class includes, for instance, the space of self-adjoint compact operators on a complex Hilbert space,
		endowed with the cone of positive-semidefinite operators 
		(see for instance \cite[Lemma~7.3]{GlueckGrohPreprint} for a proof). 
		The class also includes (the self-adjoint part of) every pre-dual of an atomic von Neumann algebra, 
		see \cite[Definition~III.5.9 and Corollary~III.5.11]{Takesaki1979}.
		
		\item 
		Every ordered Banach space on which the norm is additive on the positive cone, 
		meaning that $\norm{x+y} = \norm{x} + \norm{y}$ for all $x,y \ge 0$. 
		Indeed, the cone in such a space is clearly normal and 
		hence every increasing ordered bounded sequence $(x_n)$ is norm bounded. 
		One can then use the additivity of the norm cone to check that $(x_n)$ is a Cauchy sequence 
		and thus convergent.
		
		This class of spaces includes, for instance, all $L^1$-spaces, the space of finite measures 
		over any given $\sigma$-algebra (endowed with the total variation norm), 
		and all non-commutative $L^1$-spaces 
		(for instance, the space of self-adjoint trace class operators on a complex Hilbert space 
		with the trace norm and the cone of positive semi-definite operators).
		
		\item 
		As pointed out in Remark~\ref{rems:oc-norm}, 
		every Banach lattice with order continuous norm 
		-- for instance the space $c_0$ of real-valued sequences that converge to $0$ 
		(this class is also included in the class described in~(b)).
	\end{enumerate}
\end{examples}

Vector-valued $L^p$-spaces constitute a further class of examples that satisfy the equivalent conditions in Theorem~\ref{thm:oc-norm}: 

\begin{example}
	\label{exa:bochner}
	Let $X$ be an ordered Banach spaces, let $(\Omega,\mu)$ be a measure space, 
	and let $p \in [1,\infty)$. 
	If $X$ satisfies the equivalent conditions in Theorem~\ref{thm:oc-norm}, 
	then the same is true for the Bochner space $L^p(\Omega,\mu; X)$. 
	Here, $L^p(\Omega,\mu;X)$ is endowed with the pointwise almost everywhere order induced from $X$.
\end{example}

\begin{proof}
	We show that $L^p(\Omega,\mu;X)$ satisfies condition~\ref{thm:oc-norm:itm:sequence} in Theorem~\ref{thm:oc-norm}. 
	It suffices to do so for sequences in the positive cone only. 
	So let $(f_n)$ be an increasing sequence of functions in $L^p(\Omega,\mu;X)$ 
	and assume that there exists a function $h \in L^p(\Omega,\mu;X)$ such that $0 \le f_n \le h$ for all indices $n$. 
	After changing all the involved functions on a set of measure $0$ if necessary, 
	we may assume that the sequence $(f_n(\omega))$ in $X$ is increasing and located in the order interval $[0,h(\omega)] \subseteq X$ 
	for all $\omega \in \Omega$. 
	As $X$ satisfies condition~\ref{thm:oc-norm:itm:sequence} in Theorem~\ref{thm:oc-norm} 
	it follows that the sequence $(f_n(\omega))$ norm converges to a vector $f(\omega) \in [0, h(\omega)]$ for each $\omega \in \Omega$. 
	The function $f$ is strongly measurable as the pointwise limit of a sequence of strongly measurable functions. 
	
	Since the cone $X_+$ is normal according to Theorem~\ref{thm:oc-norm:itm:sequence}, 
	there exists a constant $C \ge 1$ such that $\norm{x} \le C \norm{y}$ for all $x,y \in X$ that satisfy $0 \le x \le y$. 
	Hence, one has $\int_\Omega \norm{f}^p \dx \mu \le C^p \int_\Omega \norm{h}^p \dx \mu < \infty$, so $f \in L^p(\Omega,\mu;X)$. 
	Moreover, $\norm{f(\omega)-f_n(\omega)}^p \le C^p \norm{h(\omega)}^p$ for all $\omega \in \Omega$, 
	so it follows from the dominated convergence theorem that $\norm{f-f_n}_{L^p} \to 0$ as $n \to \infty$.
\end{proof}

In the special case $p=1$ and $\Omega = [0,\infty)$ with the Lebesgue measure, 
the claim from Example~\ref{exa:bochner} was shown in \cite[Lemma~3.6]{ArendtDanersPreprint} by the same argument. 

It is natural to ask whether condition~\ref{thm:oc-norm:itm:net-with-min-upper-bound} in Theorem~\ref{thm:oc-norm} 
can be weakened as follows.

\begin{open_problem}
	\label{op:oc-norm}
	Let $X$ be an ordered Banach space and assume that every increasing net in $X$ that has a supremum $x \in X$ 
	is norm convergent to $x$. 
	Does it follow that $X$ satisfies the equivalent assertions of Theorem~\ref{thm:oc-norm}?
\end{open_problem}

\begin{remark}
	\label{rem:oc-norm-normal}
	The assumption of Open Problem~\ref{op:oc-norm} implies that the cone $X_+$ is normal. 
	More generally, if every increasing sequence in $X$ that has a supremum is norm convergent, 
	then $X_+$ is normal. 
	
	Indeed, if the cone is not normal, then there exists a vector $x \in X_+$ such that the order interval $[0,x]$ is not norm bounded 
	\cite[Theorem~2.40(1) and~(4) on p.\,87]{AliprantisTourky2007}.
	Thus, Lemma~\ref{lem:non-normal} below implies that there exists 
	an increasing sequence in $[0,x]$ that is not norm bounded. 
	By substracting this sequence from $x$ and then choosing an appropriate subsequence 
	we can find a decreasing sequence $(y_n)$ in $[0,x]$ such that $\norm{y_n} \ge n^2$ for each integer $n \ge 1$. 
	So the sequence $\big(\frac{1}{n} y_n\big)$, which is also decreasing, is not norm bounded; 
	in addition it has infimum $0$ since $0 \le \frac{1}{n}y_n \le \frac{1}{n} x$ for each $n$. 
	Thus, $\big(y_1 - \frac{1}{n} y_n\big)$ is an increasing sequence in $X_+$ which has a supremum but is not norm bounded 
	and is thus, in particular, not norm convergent.
\end{remark}

One way to prove that Open Problem~\ref{op:oc-norm} has a positive answer would be if one could show that, in an ordered Banach space, 
an increasing net with a minimal upper bound automatically has a smallest upper bound (i.e.\ a supremum). 
However, this is not true, as the following example shows. 

\begin{example}
	\label{exa:disk-algebra}
	Let $A(\bbD)$ denote the \emph{disk algebra}, i.e.\ the space of continuous complex valued functions 
	on the closed complex unit disk $\overline{\bbD}$ that are holomorphic on the open unit disk $\bbD$. 
	This is a complex Banach space with respect to the supremum norm. 
	
	Now let $X \subseteq A(\bbD)$ denote the space of those function in $A(\bbD)$ that are real-valued on $[-1,1]$ 
	and let the cone $X_+$ be the set of those functions that map $[-1,1]$ into $[0,\infty)$. 
	It follows from the uniqueness theorem for holomorphic functions that a function $f$ that satisfies $\pm f \in X_+$ is $0$, 
	so $X_+$ is pointed, i.e.\ a cone. 
	The cone $X_+$ is generating since it contains the constant function $\one$, 
	but one can easily check that $X_+$ is not normal 
	-- for instance by considering the sequence $\big(\one + \sin(n \argument)\big)$ which is order bounded but not norm bounded in $X$. 
	 
	It follows from a monotone version of the Weierstraß approximation theorem 
	-- see Lemma~\ref{lem:weierstrass-monotone} below -- 
	that there exists an a increasing sequence of polynomials $(p_n)$ in $X_+$ such that 
	$p_n(x) \to 2 - \modulus{x}$ for all $x \in [-1,1]$ 
	(and the convergence is even uniform on $[-1,1]$ 
	-- but note that the sequence $(p_n)$ is not norm bounded in $X$ since otherwise 
	the limit would be holomorphic by Vitali's theorem). 
	
	It follows from the uniqueness theorem for holomorphic functions that both functions $f,g \in X_+$ 
	given by 
	\begin{align*}
		f(z) = 2+z 
		\qquad \text{and} \qquad 
		g(z) = 2-z
	\end{align*}
	for all $z \in \overline{\bbD}$ are mininal upper bounds of the sequence $(p_n)$ in $X_+$. 
	Since there are two different minimal upper bounds, the sequence does not have a smallest upper bound.
\end{example}

Variations of the space $X$ in Example~\ref{exa:disk-algebra} are very useful to construct counterexamples 
in the theory of ordered Banach spaces and positive operators 
-- see for instance the classical example  
by Bonsall for a positive operator whose spectral radius is not in the spectrum \cite[Example~(iv) on pp.\,57--58]{Bonsall1958}
and a recent semigroup adaptation thereof in \cite[Examples~2.3]{GlueckMironchenko2024}.
In the preceding example we needed the following monotone version of the Weierstraß approximation theorem.

\begin{lemma}
	\label{lem:weierstrass-monotone}
	Let $g: [-1,1] \to \bbR$ be continuous. Then there exists a sequence of polynomial functions $(p_n)$ on $\bbR$ 
	that is increasing on $[-1,1]$ with respect to the pointwise order and that converges uniformly to $g$ on $[-1,1]$.
\end{lemma}

\begin{proof}
	We first make the following observation $(*)$:
	If $\varepsilon > 0$, then there exists a polynomial function $p$ on $\bbR$ 
	such that $g-\varepsilon \le p \le g - \frac{\varepsilon}{2}$ on $[-1,1]$. 
	This follows by applying the Weierstraß approximation theorem for real-valued continuous functions on $[-1,1]$ 
	to the function $g-\frac{3}{4}\varepsilon$.
	
	Now one can apply $(*)$ to each of the numbers $\varepsilon_n := \frac{1}{2^n}$ to get a sequence of polynomials $(p_n)$ 
	that satisfy
	\begin{align*}
		p_n \le g - \frac{\varepsilon_n}{2} = g - \varepsilon_{n+1} \le p_{n+1}
	\end{align*}
	on $[-1,1]$ for each $n \ge 0$.
\end{proof}

Let us note once again that the cone $X_+$ in Example~\ref{exa:disk-algebra} is not normal. 
We do not know an example of an increasing net $(x_j)$ in an ordered Banach space with normal cone 
such that $(x_j)$ has a minimal upper bound but not a smallest upper bound. 
If such an example does not exist, then it follows from Remark~\ref{rem:oc-norm-normal} 
that Open Problem~\ref{op:oc-norm} has a positive answer.

Generally speaking, spaces $X$ where one can find nets 
that have a minimal upper bound but not a smallest upper bound, are somewhat special.
In each space from the following list such a net does not exist.

\begin{examples}
	Let $X$ be an ordered Banach space (or, more generally than that, an Archimedean ordered vector space). 
	Any of the following conditions implies that every increasing net in $X$ with a minimal upper bound has a supremum:
	\begin{enumerate}[label=(\alph*)]
		\item 
		The space $X$ is a vector lattice.

		\item 
		Every increasing net that is order bounded from above has a supremum. 
		
		This is, for instance, true in the dual space $X'$ of any ordered Banach space $X$ with generating cone 
		(and more generally in the dual space of any pre-ordered Banach space with generating wedge).
		Moreover, the condition holds in all spaces that satisfy the equivalent assumptions of Theorem~\ref{thm:oc-norm} 
		-- and thus, in particular, in all of the spaces listed in Examples~\ref{exas:almost-oc-norm}. 
				
		\item 
		The cone $X_+$ is generating and the space $X$ is \emph{pervasive}, 
		i.e.\ for every function $f \in X$ that does not satisfy $f \le 0$, 
		there exists a vector $0 \not= u \in X_+$ 
		such that every upper bound of $\{0,f\}$ is also an upper bound of $u$.
		
		Indeed, consider such a space $X$. 
		It suffices to prove that for every decreasing net $(x_j)$ with maximal lower bound $0$, 
		the vector $0$ is even the largest lower bound. 
		So take such a net and let $f \in X$ be a lower bound of $(x_j)$. 
		If $f \not\le 0$, then there exists a vector $u$ as in the definition of pervasiveness. 
		Since every $x_j$ dominates both $f$ and $0$, it also dominates $u$. 
		Since $u \gneq 0$, this is a contradiction to $0$ being a maximal lower bound of $(x_j)$.
		Hence, $f \le 0$, so $0$ is indeed the largest lower bound.
	\end{enumerate}
\end{examples}

The notion of \emph{pervasiveness} for Archimedean spaces with generating cone 
-- or, more generally, for so-called \emph{pre-Riesz spaces} -- 
was introduced in \cite{vanGaansKalauch2008} and is discussed in detail in \cite{KalauchMalinowski2019b} 
and in various places throughout the book \cite{KalauchVanGaans2019}.

There is also the notion of a \emph{weakly pervasive} space $X$, 
which means that two vectors $f,g \in X_+$ that satisfy $[0,f] \cap [0,g] = \{0\}$ 
always have an infimum and the infimum equals $0$; 
see \cite[Definition~8 and Lemma~9]{KalauchMalinowski2019b}. 
This property is, as suggested by its name, implied by pervasiveness. 
However, in contrast to pervasiveness, weak pervasiveness does not imply that every increasing net in $X$ 
with a minimal upper bound has a supremum. 
Indeed, the space $X$ in Example~\ref{exa:disk-algebra} is weakly pervasive for the trivial reason 
that there are no non-zero vector $f,g \in X_+$ which satisfy $[0,f] \cap [0,g] = \{0\}$.  
Indeed, for two non-zero vectors $f,g \in X_+$, say of norm $\le 1$, the product $fg$ is contained in $[0,f] \cap [0,g]$ 
and is non-zero by the uniqueness theorem for holomorphic functions.

\section{A theorem on separable spaces}
\label{sec:separable}

If a countably order omplete Banach lattice $X$ is separable, 
then $X$ has order continuous norm. 
This is a classical result in Banach lattice theory (see for instance \cite[the Corollary to Theorem~5.14 on pp.\,94-95]{Schaefer1974}). 
Its proof is based on constructing a certain disjoint sequence in $X$ and thereby showing that, 
if $X$ does not have order continuous norm, then it contains a copy of $\ell^\infty$ as a sublattice, 
see e.g.\ \cite[Lemma~5.13 and Theorem~5.14 on pp.\,92--94]{Schaefer1974}). 

We will now show that the same result remains true in the setting of ordered Banach spaces. 
In this general setting, though, there is no chance that a proof based on disjoint sequences will work. 
While there is a sensible concept of disjointness available in ordered Banach spaces with generating cone and, 
more generally, in so-called \emph{pre-Riesz spaces} -- see \cite{vanGaansKalauch2006} and \cite[Section~4.1]{KalauchVanGaans2019} --, 
disjointness of two vectors tends to be a very strong property in this setting. 
For instance, there exist plenty of spaces which are so-called \emph{anti-lattices}, 
meaning that two non-zero vectors in the cone are never disjoint \cite[Section~4.1.4]{KalauchVanGaans2019}.
Hence, a different argument is needed to prove the following theorem.

\begin{theorem}
	\label{thm:separable}
	Let $X$ be an ordered Banach space with normal cone 
	and assume that every order bounded increasing sequence in $X$ has a supremum. 
	If $X$ is separable, then every order bounded increasing sequence in $X$ is norm convergent 
	(i.e., the equivalent assertions of Theorem~\ref{thm:oc-norm} are satisfied).
\end{theorem}

\begin{proof}
	Due to the normality of the cone we may, and shall, assume that the norm is monotone on $X_+$, 
	meaning that $0 \le x \le y$ implies $\norm{x} \le \norm{y}$ for all $x,y \in X$; 
	see \cite[Theorem~2.38(1) and~(2)]{AliprantisTourky2007}.
	Suppose that there exists an increasing sequence $(x_n)$ in $X_+$ that is not norm convergent 
	but is order bounded by an element $v \in X_+$.
	We have to show that $X$ is not separable.
	\smallskip 
	
	\emph{Step~1:} 
	We claim: 
	there exist a sequence $(z_k)$ in $X_+$ and a vector $v \in X_+$ such that $\norm{z_k} \ge 1$ for each $k$ 
	and $\sum_{k = 1}^n z_k \le v$ for each $n \in \bbN$.
		
	To see this, first set $y_1 \coloneqq x_1 \in X_+$ and $y_k \coloneqq x_k - x_{k-1} \in X_+$ for each $k \ge 2$. 
	Then $\sum_{k=1}^n y_k = x_n \le v$ for each $n \in \bbN$. 
	Since the sequence $(\sum_{k=1}^n y_k) = (x_n)$ is not Cauchy, 
	we can find a number $\delta > 0$ and a sequence of non-empty, finite, pairwise disjoint sets $F_k \subseteq \bbN$ 
	such that $z_k \coloneqq \sum_{j \in F_k} y_j$ has norm $\ge \delta$ for each $k \in \bbN$.
	Clearly, $\sum_{k=1}^n z_k \le v$ for all $n \in \bbN$. 
	By replacing each $z_k$ with $\frac{z_k}{\delta}$ and $v$ with $\frac{v}{\delta}$ 
	we obtain the desired properties.
	\smallskip
	
	\emph{Step~2:} We construct infinite series by using suprema: 
	
	For each $A \subseteq \bbN$ we define the expression $\sum_{k \in A} z_k$ 
	as the supremum of the increasing sequence $\big(\sum_{A \ni k \le n} z_k\big)_{n \in \bbN}$ 
	which exists by the assumption of the theorem as the sequence is order bounded by $v$.
	Note that $\sum_{k \in A} z_k$ is the supremum of $\big\{ \sum_{k \in F} z_k \suchthat F \subseteq A \text{ finite} \big\}$.
	From this one can easily derive that 
	$$\sum_{k \in A} z_k + \sum_{k \in B} z_k = \sum_{k \in A \cup B} z_k$$
	for all disjoint $A, B \subseteq \bbN$.
	\smallskip
	
	\emph{Step~3:} 
	We claim that 
	there exists an infinite set $M \subseteq \bbN$ with the following property:
	whenever $A,B \subseteq M$ are disjoint and at least one of them is infinite, 
	then $\norm{\sum_{k \in A} z_k - \sum_{k \in B} z_k} \ge \frac{1}{2}$. 
	
	Assume the contrary.
	Then we can recursively construct two sequences of sets $(A_n)$ and $(B_n)$ with the following properties:
	each $A_n$ is infinite and for each $n \in \bbN$ one has $A_{n+1}, B_{n+1} \subseteq A_n$, 
	as well as $A_n \cap B_n = \emptyset$ and 
	\begin{align*}
		\Big\lVert
			\underbrace{\sum_{k \in A_n} z_k}_{\eqqcolon a_n}
			- 
			\underbrace{\sum_{k \in B_n} z_k}_{\eqqcolon b_n}
		\Big\rVert 
		< 
		\frac{1}{2}
	\end{align*} 
	Each $a_n$ dominates one of the vectors $z_k$ and thus satisfies $\norm{a_n} \ge 1$ as the norm on $X$ is monotone. 
	Hence,
	\begin{align*}
		\norm{a_n + b_n} 
		\ge 
		2\norm{a_n} 
		- 
		\norm{a_n - b_n}
		\ge 
		2 \norm{a_n} 
		- 
		\frac{1}{2}
		\ge 
		\frac{3}{2} \norm{a_n}
	\end{align*}
	for each $n \in \bbN$. 
	For all $n$ the sets $A_{n+1}$ and $B_{n+1}$ are disjoint and contained in $A_n$, so one has $a_n \ge a_{n+1} + b_{n+1}$. 
	Therefore,
	\begin{align*}
		\norm{a_n} 
		\ge 
		\norm{a_{n+1} + b_{n+1}} 
		\ge 
		\frac{3}{2}
		\norm{a_{n+1}}
	\end{align*}
	and hence $\norm{a_1} \ge \big(\frac{3}{2}\big)^n$ for each $n \in \bbN$, which is a contradiction.
	\smallskip 
	
	\emph{Step~4:} We show that $X$ is not separable.
	
	To this end we take the set $M \subseteq \bbN$ whose existence we proved in Step~3 and 
	introduce the following equivalence relation $\sim$ on its power set $2^M$: 
	for all $C,D \in 2^M$, define $C \sim D$ if and only if the symmetric difference $C \mathbin{\Delta} D$ is finite. 
	Each of the equivalence classes is countable. 
	Since $2^M$ is uncountable, there exist uncountably many equivalence classes. 
	
	Select one representative of each equivalence class and collect them in a set $P \subseteq 2^M$.
	Consider $C,D \in P$ such that $C \not= D$. 
	Then $C \mathbin{\Delta} D$ is infinite 
	and hence, at least one of the differences $C \setminus D$ and $D \setminus C$ is infinite.
	We conclude that 
	\begin{align*}
		\norm{\sum_{k \in C} z_k - \sum_{k \in D} z_k} 
		= 
		\norm{\sum_{k \in C \setminus D} z_k - \sum_{k \in D \setminus C} z_k}
		> 
		\frac{1}{2}
		,
	\end{align*}
	where the equality follows from the observation at the end of Step~2 
	and the inequality follows from the properties of $M$ given in Step~3. 
	So $\Big( \sum_{k \in C} z_k \Big)_{C \in P}$ is an uncountable family of elements of $X$ 
	that all have mutual distance $> \frac{1}{2}$.
	Hence, $X$ is not separable.
\end{proof}

Recall from Theorem~\ref{thm:oc-norm} that the conclusion of Theorem~\ref{thm:separable} implies normality of the cone. 
Hence, it is natural to ask whether the assumption in Theorem~\ref{thm:separable} that the cone be normal is in fact redundant.
We pose this and a more general version of this question as an open problem:

\begin{open_problem}
	\label{op:sup-implies-normal}
	\leavevmode
	\begin{enumerate}
		\item 
		If every increasing order bounded sequence in an ordered Banach space $X$ has a supremum, 
		does it follow that $X$ is normal?
		
		\item 
		Is the assumption that $X_+$ be normal redundant in Theorem~\ref{thm:separable}?
	\end{enumerate}
\end{open_problem}

Clearly, if part~(1) of Open Problem~\ref{op:sup-implies-normal} has a positive answer, then so has part~(2).
In an attempt to prove~(1) a natural approach is trying to adapt the argument that the Levi property 
implies normality of the cone, which can for instance be found in~\cite[Theorem~2.45]{AliprantisTourky2007}. 
However, the author has not been able to find a variant of this argument that works to prove~(1).

We give an application of Theorem~\ref{thm:separable} in Corollary~\ref{cor:compact-operators-oc-norm-non-reflexive} 
in the next section. 
In the present section, we merely discuss a simple toy example 
to demonstrate how Theorem~\ref{thm:separable} can be used (Example~\ref{exa:concave} below). 
To decrease the technical overhead in the example the following corollary is useful.

\begin{corollary}
	\label{cor:tau}
	Let $X$ be an ordered Banach space with normal cone and assume that $X$ is separable. 
	Let $\tau$ be a vector space topology on $X$ with the following two properties:
	\begin{enumerate}[label=\upshape(\arabic*)]
		\item 
		The cone $X_+$ is closed with respect to $\tau$. 
		
		\item 
		Every increasing norm bounded sequence in $X_+$ converges to a vector in $X$ with respect to $\tau$. 
	\end{enumerate}
	Then every increasing norm bounded sequence in $X$ is norm convergent.
\end{corollary}

\begin{proof}
	First make the following observation $(*)$: 
	If $(x_n)$ is an increasing norm bounded sequence in $X_+$, 
	then it $\tau$-converges to a vector $x \in X$ according to~(2), 
	and it follows from~(1) that $x$ is the supremum of $(x_n)$. 
	
	Now one can apply~$(*)$ twice: 
	As $X_+$ is normal, every order bounded sequence in $X$ is norm bounded; 
	thus, $(*)$ shows that the assumptions of Theorem~\ref{thm:separable} are satisfied. 
	The theorem yields that every increasing order bounded sequence in $X_+$ is norm convergent. 
	On the other hand, $(*)$ shows that every increasing norm bounded sequence is order bounded, 
	so the claim is proved.
\end{proof}

\begin{example}
	\label{exa:concave}
	For each integer $n \ge 0$ let $f_n: [0,1] \to [0,\infty)$ be a continuous and concave function 
	and assume that there exists a number $M \ge 0$ such that $\sum_{n=0}^\infty f_n(s) \le M$ for each $s \in [0,1]$. 
	Then the series $\sum_{n=0}^\infty f_n$ converges in sup norm.
\end{example}

\begin{proof}
	Let $X = C([0,1])$ denote the space of all continuous real-valued functions on $[0,1]$, 
	endowed with the sup norm.
	This is a separable Banach space. 
	We now endow $X$ with the closed cone 
	\begin{align*}
		X_+ := \{f \in C([0,1]) \suchthat f \text{ is concave and } f(s) \ge 0 \text{ for all } s \in [0,1] \},
	\end{align*}
	which turns $X$ into an ordered Banach space. 
	In order to apply Corollary~\ref{cor:tau}, 
	let $\tau$ denote the topology on $X$ of pointwise convergence on the open set $(0,1)$ 
	-- i.e., let $\tau$ be the initial topology of the family $(\delta_x)_{x \in (0,1)}$ of point evaluations.
	
	Assumption~(1) of the corollary is satisfied since $X$ only contains continuous functions.
	To see that assumption~(2) is satisfied, let $(g_n)$ be an increasing and norm bounded sequence in $X_+$. 
	Then $(g_n)$ is, in particular, increasing with respect to the pointwise order. 
	So the sequence converges pointwise on $[0,1]$ to a concave and bounded function $\tilde g: [0,1] \to \bbR$; 
	but $\tilde g$ need not be continuous on $[0,1]$ and thus need not be in $X$. 
	However, due to the concavity $\tilde g$ is continuous in the interior $(0,1)$ and its restriction $\tilde g|_{(0,1)}$ 
	can be extended to a continuous function $g: [0,1] \to \bbR$; so $g \in X$.
	Observe that $(g_n)$ is $\tau$-convergent to $g$.
	So Corollary~\ref{cor:tau} is applicable and gives the claim. 
\end{proof}

Obviously, a similar arguments works for convex instead of concave functions, too.

\section{Dini's theorem revisited}
\label{sec:dini}

A classical form of Dini's theorem says the following: 
if $K$ is a compact Hausdorff space, $(f_n)$ is a pointwise increasing sequence (or net) 
of continuous real-valued functions on $K$ and the pointwise limit $f$ of $(f_n)$ is real-valued 
and continuous, then the convergence is automatically uniform. 
There is a version of the theorem for ordered Banach spaces which reads as follows:
if $X$ is an ordered Banach space with normal cone 
and $(x_j)$ is an increasing and weakly convergent net in $X$, 
then $(x_j)$ is even norm convergent. 
This result can, for instance, be found in \cite{Weston1957} (for sequences), in \cite[Theorem~IV.3.1]{Wulich2017}, 
or, in the language of filters, in \cite[paragraph~4.3 on p.\,223 in Section~V.4]{SchaeferWolff1999}. 
Since increasing sequences and nets are the main topic of the present paper, 
we consider it worthwhile to point out that the following slightly more general version of the theorem holds 
(although the proof is almost the same) 
and to list a few examples. 
As an application we will give a sufficient criterion in Corollary~\ref{cor:compact-operators-oc-norm} 
for the space of all compact operators between two Banach lattices 
to satisfy the equivalent conditions of Theorem~\ref{thm:oc-norm}.

\begin{theorem}
	\label{thm:dini}
	Let $X$ be an ordered Banach space and  
	let $C' \subseteq X'_+$ be a set of positive functionals with the following properties: 
	\begin{enumerate}[label=\upshape(\arabic*)]
		\item 
		$C'$ is weak*-compact. 
		
		\item 
		$C'$ \emph{determines positivity} in the following sense: 
		if $x \in X$ and $\langle c', x \rangle \ge 0$ for all $x' \in C'$, then $x \in X_+$.
		
		\item 
		$C'$ \emph{is almost norming on $X_+$} in the following sense: 
		there is a number $\delta > 0$ such that 
		$\sup_{c' \in C'} \langle c', x \rangle \ge \delta \norm{x}$ for every $x \in X_+$.
	\end{enumerate}
	If $(x_j)_{j \in J}$ is an increasing net in $X$ and $x_0 \in X$ 
	such that $\langle c', x_j \rangle \to \langle c', x_0 \rangle$ for each $c' \in C'$, 
	then $x_j \to x_0$ in norm.
\end{theorem}

\begin{proof}
	Let $(x_j)_{j \in J}$ and $x_0$ be as in the theorem. 
	Then $\langle c', x_j \rangle \le \langle c', x_0 \rangle$ for all $c' \in C'$ 
	and it thus follows from assumption~(2) that $x_j \le x_0$ for all $j \in J$.
	
	To show the norm convergence, fix a number $\varepsilon > 0$.
	For each $j \in J$ the set $C'_j := \{c' \in C' \suchthat \langle c', x_0-x_j \rangle < \varepsilon \}$
	is weak* open within $C'$ and it follows from the assumption that $\bigcup_{j \in J} C'_j = C'$. 
	As $C'$ is weak* compact according to~(1), we can thus find finitely many indices $j_1, \dots, j_n \in J$ 
	such that $C'_{j_1} \cup \dots \cup C'_{j_n} = C'$. 
	
	Choose an index $j_0 \in J$ that dominates all the indices $j_1, \dots, j_n$. 
	If $j \ge j_0$ and $c' \in C'$ we have $c' \in C'_{j_k}$ for some $k \in \{1, \dots, n\}$; 
	since $j \ge j_k$ and since the net under consideration is increasing, 
	one has $x_j \ge x_{j_k}$ and thus  
	\begin{align*}
		\langle c', x_0-x_j \rangle 
		\le 
		\langle c', x_0-x_{j_k} \rangle 
		< 
		\varepsilon
		.
	\end{align*}
	As we know from the beginning of the proof that $x_0-x_j \in X_+$, 
	we conclude from assumption~(3) that $\norm{x_0-x_j} \le \frac{\varepsilon}{\delta}$ for all $j \ge j_0$.
\end{proof}

For the sake of completeness let us state the following well-known special case 
of Theorem~\ref{thm:dini} explicitly.

\begin{corollary}
	\label{cor:dini-weak-conv} 
	Let $X$ be an ordered Banach space with normal cone 
	and let $(x_j)_{j \in J}$ be an increasing net in $X$ 
	that converges weakly to a point $x_0 \in X$. 
	Then one even has $x_j \to x_0$ in norm.
\end{corollary}

\begin{proof}
	This follows by choosing the set $C'$ in Theorem~\ref{thm:dini} to be the positive closed unit ball in $X'$. 
	Indeed, this set $C'$ satisfies~(1) by the Banach--Alaoglu theorem 
	and it satisfies~(2) as a consequence of the Hahn--Banach separation theorem. 
	To see~(3), note that the dual wedge $X'_+$ is generating in $X'$ 
	since $X_+$ is normal \cite[Theorem~4.5]{KrasnoselskijLifshitsSobolev1989}. 
	Hence, property~(3) follows from the uniform decomposition theorem in spaces with a generating closed wedge 
	\cite[Proposition~1.1.2]{BattyRobinson1984}.
\end{proof}

The relation of Theorem~\ref{thm:dini} to Dini's classical theorem for continuous functions 
is explained in the following example and the subsequent remark.

\begin{example}[Dini's theorem on spaces of continuous functions]
	\label{exa:dini-continuous-functions}
	\leavevmode
	\begin{enumerate}[label=(\alph*)]
		\item\label{exa:dini-continuous-functions:itm:compact}
		Let $K$ be a compact Hausdorff space and endow 
		the space $C(K)$ of real-valued continuous functions on $K$ 
		with the sup norm and with the pointwise order.
		Let $(f_j)$ be an increasing net in $C(K)$ 
		and assume that $(f_j)$ converges pointwise to a function $f \in C(K)$. 
		Then the convergence takes place in norm. 
		
		This well-known result is a special case of Theorem~\ref{thm:dini}: 
		just apply the theorem to the set $C' := \{\delta_s \suchthat s \in K\} \subseteq C(K)'$ of Dirac measures $\delta_s$ on $K$.
		
		\item\label{exa:dini-continuous-functions:itm:locally-compact}
		More generally, consider a locally compact Hausforff space $L$ 
		and endow with space $C_0(L)$ of real-valued continuous functions on $L$ that vanish at infinity 
		with the pointwise order and the sup norm.
		If $(f_j)$ is an increasing net in $C_0(L)$ 
		that converges pointwise to a function $f \in C_0(L)$, 
		then the convergence takes place in norm. 
		
		To see this, one can apply~\ref{exa:dini-continuous-functions:itm:compact} 
		to the one-point compactification of $L$. 
		Alternatively, one can apply Theorem~\ref{thm:dini} 
		to the weak* compact set $C' := \{\delta_s \suchthat s \in L\} \cup \{0\} \subseteq C_0(L)'$.
	\end{enumerate}
\end{example}

\begin{remarks}
	\begin{enumerate}
		\item 
		For sequences, Dini's classical theorem for continuous functions 
		(see Example~\ref{exa:dini-continuous-functions}\ref{exa:dini-continuous-functions:itm:compact}) 
		can also be obtained from Corollary~\ref{cor:dini-weak-conv}. 
		Indeed, if $K$ is a compact Hausdorff space and $(f_n) \subseteq C(K)$ is an increasing sequence 
		that converges pointwise to a continuous function $f \in C(K)$, 
		then it follows from the Riesz representation theorem for the dual space $C(K)'$ 
		and from the monotone convergence theorem for integrals that $(f_n)$ convergences even weakly to $f$. 
		Hence, Corollary~\ref{cor:dini-weak-conv} gives the norm convergence.
		
		Yet, this argument cannot be applied to nets since the monotone convergence theorem 
		does not hold for nets. 
		This is one reason (among others, see below) 
		why it seems more natural to formulate Theorem~\ref{thm:dini} in the more general version 
		that uses the set $C'$.
		
		\item 
		On the other hand, one can also derive Theorem~\ref{thm:dini} 
		from Dini's classical theorem for continuous functions 
		(Example~\ref{exa:dini-continuous-functions}\ref{exa:dini-continuous-functions:itm:compact}):
		In the situation of the theorem, consider the continuous linear map 
		\begin{align*}
			\phi: X & \to C(C'), \\ 
			x & \mapsto \big( c' \mapsto \langle c', x \rangle \big),
		\end{align*}
		where $C'$ is endowed with the weak* topology and is hence compact by assumption~(1).
		This maps is positive since $C' \subseteq X'_+$ and it is even bipositive 
		(i.e.\ a vector $x \in X$ is positive if and only if $\phi(x)$ is positive) 
		due to assumption~(2) in the theorem. 
		Moreover, it follows from assumption~(3) that the map is bounded below.
		The assumption on $(x_j)_{j \in J}$ implies that $\big(\phi(x_j)\big)_{j \in J}$ converges pointwise to $\phi(x_0)$, 
		so by Dini's classical theorem this convergence takes place with respect to the sup norm in $C(C')$. 
		Since $\phi$ is bounded below, it follows that $(x_j)_{j \in J}$ converges in norm to $x_0$.
	\end{enumerate}
\end{remarks}

As slightly less obvious example is the following.

\begin{example}[Monotone convergence of self-adjoint operators]
	\label{exa:dini-self-adjoint}
	Let $H$ be a complex Hilbert space. 
	\begin{enumerate}[label=(\alph*)]
		\item\label{exa:dini-self-adjoint:itm:compact}
		Endow the space $\calK(H)_{\operatorname{sa}}$ of self-adjoint compact linear operators on $H$ 
		with the cone of positive semi-definite elements. 
		Let $(A_j)$ be an increasing net in $\calK(H)_{\operatorname{sa}}$ and let $A \in \calK(H)_{\operatorname{sa}}$. 
		If $\inner{x}{A_j x} \to \inner{x}{A x}$ for each $x \in H$, 
		then $A_j \to A$ with respect to the operator norm.
		
		To see this, let $B_H$ denote the closed unit ball in $H$. 
		and first note that, for every compact linear operator $C$ on $H$, 
		the mapping $B_H \times B_H \to \bbC$, $(x,y) \mapsto \inner{x}{Cy}$ is jointly continuous 
		with respect to the weak topology on $H$.
		Indeed, this is a consequence of the observation that, if a bounded net $(x_j)$ in $H$ 
		converges weakly to a point $x \in H$, 
		then the net of functionals $\big(\inner{x_j}{C \argument}\big)$ on $H$ 
		converges in norm to the functional $\inner{x}{C \argument}$ on $H$.%
		
		Now we can show the claimed operator norm convergence of $(A_j)$. 
		For each $x \in B_H$ consider the bounded linear functional 
		$x \otimes x: \calK(H)_{\operatorname{sa}} \to \bbR$ that is given by 
		$\langle x \otimes x, B \rangle := \inner{x}{Bx}$ for all $B \in \calK(H)_{\operatorname{sa}}$. 
		We define $C' \coloneqq \{x \otimes x \suchthat x \in B_H\}$ and show that $C'$ 
		satisfies the assumptions~(1)--(3) from Theorem~\ref{thm:dini}:
		
		(1) 
		As $C'$ is norm bounded, we only need to check that it is weak* closed. 
		So let $(x_j \otimes x_j)$ be a net in $C'$ that converges weak* 
		to functional $\varphi$ on $\calK(H)_{\operatorname{sa}}$.
		After switching to a subnet we may, and shall, assume that $(x_j)$ is weakly convergent to a point $x \in B_H$. 
		For every $C \in \calK(H)_{\operatorname{sa}}$ we then obtain, according to the preceding paragraph, 
		\begin{align*}
			\langle x_j \otimes x_j, C \rangle 
			= 
			\inner{x_j}{Cx_j} 
			\to 
			\inner{x}{Cx} 
			= 
			\langle x \otimes x, C \rangle
			,
		\end{align*}
		and thus $\varphi = x \otimes x \in C'$. 
		
		(2) 
		This is obvious from the definition of positive elements in $\calK(H)_{\operatorname{sa}}$.
		
		(3)
		One has $\sup_{x \in B_H} \langle x \otimes x, C \rangle = \norm{C}$ 
		for every positive $C \in \calK(H)_{\operatorname{sa}}$ due to the self-adjointness of $C$. 
		
		Hence, Theorem~\ref{thm:dini} is applicable and thus gives the claimed norm convergence of $A_j$ to $A$.

		\item\label{exa:dini-self-adjoint:itm:compact-order-interval}
		Under the assumptions in~\ref{exa:dini-self-adjoint:itm:compact} the net $(A_j)$ is eventually contained in an order interval.
		Hence, instead of using Theorem~\ref{thm:dini}, one can also obtain the operator norm convergence of $(A_j)$ to $A$
		as a special case of the following observation: 
		
		Let $A,B,C \in \calK(H)_{\operatorname{sa}}$ and assume that $(A_j)$ is a net in the order interval $[B,C]$ 
		such that $\inner{x}{A_j x} \to \inner{x}{Ax}$ for all $x \in H$. 
		Then even $A_j \to A$ with respect to the operator norm. 
		
		This can be seen by combining the follow two observations: 
		On the one hand, it follows from the polarization identity that $A_j \to A$ with respect to the weak operator topology. 
		On the other hand, every order interval in $\calK(H)_{\operatorname{sa}}$ as compact with respect to the operator norm
		(see for instance \cite[Lemma~7.3]{GlueckGrohPreprint} for a proof).

		\item\label{exa:dini-self-adjoint:itm:non-compact} 
		The assertion of part~\ref{exa:dini-self-adjoint:itm:compact} does not remain true, 
		in general, if we consider bounded instead of compact linear operators. 
		
		For instance, let $H = \ell^2$ and let $A_n$ be the multiplication with the indicator function 
		of $\{n,n+1,n+2, \dots\}$ for each $n \in \bbN$. 
		Then $\inner{x}{A_n x} \to 0$ for each $x \in \ell^2$, but the sequence $(A_n)$ is not norm convergent to $0$.
	\end{enumerate}
\end{example}

We conclude this section with two corollaries that rely on an application of Theorem~\ref{thm:dini} 
to the operator norm convergence of a net of compact operators. 
For two ordered Banach spaces $X$ and $Y$ we endow the spaces $\calL(X;Y)$ of bounded linear operators 
and $\calK(X;Y)$ of compact linear operators with the wedge of positive operators in the given space, respectively. 
If $X_+$ is total, then those wedges are even cones and thus those two operator spaces become ordered Banach spaces
(note that the order on those spaces is completely different than the order in Example~\ref{exa:dini-self-adjoint}).

\begin{corollary}
	\label{cor:compact-operators-weak-norm}
	Let $X,Y$ be ordered Banach spaces and assume that $X_+$ is generating and $Y_+$ is normal. 
	Also assume that the space $X$ is reflexive.
	Let $T$ be an element of the space $\calK(X;Y)$ of compact linear operators from $X$ to $Y$ 
	and let $(T_j)$ be an increasing net in $\calK(X;Y)$ that converges to $T$ 
	with respect to the weak operator topology. 
	Then the convergence actually takes place in operator norm.
\end{corollary}

\begin{proof}
	For all $x \in X$ and $y' \in X'$ let $\varphi_{x,y'} \in \calK(X;Y)'$ denote the functional given by 
	\begin{align*}
		\left \langle \varphi_{x,y'} , S \right \rangle_{\langle \calK(X,Y)', \calK(X,Y) \rangle}
		\coloneqq 
		\langle y', Sx \rangle_{\langle Y', Y \rangle}
	\end{align*}
	for all $S \in \calK(X;Y)$.
	The set $C' \coloneqq \{\varphi_{x,y'} \suchthat x \in X_+, \, y' \in Y'_+, \, \norm{x} \le 1, \, \norm{y'} \le 1\}$ 
	is compact with respect to the weak* topology on $\calK(X;Y)'$ since the positive unit ball in $Y'$ is weak* compact, 
	the positive unit ball in $X$ is weakly compact due to the reflexivity of $X$, 
	and every compact operator maps bounded weakly convergent nets to norm convergent nets.
	Next we observe that there exists a number $\delta > 0$ such that
	\begin{align*}
		\sup_{c' \in C'} \big\langle c', S \big\rangle_{\langle \calK(X;Y)', \calK(X;Y) \rangle} 
		\ge  
		\delta \norm{S}
	\end{align*}	
	for all $0 \le S \in \calK(X;Y)$. 
	This follows from uniform decomposition theorem in spaces with generating closed wedges \cite[Proposition~1.1.2]{BattyRobinson1984}
	since both $X_+$ and $Y'_+$ are generating 
	(the fact that $Y_+'$ is generating follows from the normality of $Y_+$ \cite[Theorem~4.5]{KrasnoselskijLifshitsSobolev1989}).
	As the set $C'$ can be checked to determine positivity of operators in $\calK(X;Y)$, 
	Theorem~\ref{thm:dini} is applicable and gives the claim.
\end{proof}

Observe that there is only one step in the previous proof where we used that we worked with compact operators 
(rather than operators that are merely bounded): 
for the weak* compactness of the set $C'$ one needs that a compact operator maps bounded weakly convergent nets 
to norm convergent nets 
(since, on a Banach space $X$, the dual pairing $\langle \argument, \argument \rangle_{\langle X', X \rangle}$ 
is not jointly continuous with respect to the weak* topology on $X'$ and the weak topology on $X$).

\begin{corollary}
	\label{cor:compact-operators-oc-norm}
	Let $X,Y$ be Banach lattices, where $X$ is reflexive and $Y$ has order continuous norm. 
	Then in the ordered Banach space $\calK(X;Y)$ of all compact linear operators from $X$ to $Y$ 
	every increasing order bounded sequence is norm convergent,  
	i.e.\ the space $\calK(X;Y)$ satisfies the equivalent conditions of Theorem~\ref{thm:oc-norm}.
\end{corollary}

\begin{proof}
	It suffices to check condition~\ref{thm:oc-norm:itm:sequence} in the theorem for increasing sequences in the positive cone. 
	So let $(T_n) \subseteq \calK(X;Y)$ be an increasing sequence such that $0 \le T_n \le S$ 
	for an operator $S \in \calK(X;Y)$ and all indices $n$. 
	For every $x \in X_+$ the sequence $(T_nx)$ in $Y_+$ is increasing and bounded from above by $Sx$, 
	so the sequence converges in norm to a vector in $[0,Sx]$. 
	From this one easily gets that $(T_n)$ converges strongly to a bounded linear operator $T: X \to Y$ 
	such that $0 \le T \le S$. 
	
	As both $X'$ and $Y$ have order continuous norm, 
	compactness of positive operators is inherited under domination, 
	see e.g.\ \cite[Theorem~16.20]{AliprantisBurkinshaw1985} or \cite[Theorem~3.7.13]{MeyerNieberg1991}.
	Hence, $T \in \calK(X;Y)$. 
	Since $X$ is reflexive we can now apply Corollary~\ref{cor:compact-operators-weak-norm} 
	to obtain operator norm convergence of $(T_n)$ to $T$. 
	Thus, $\calK(X;Y)$ does indeed satisfy the equivalent conditions of Theorem~\ref{thm:oc-norm}.
\end{proof}

Interestingly, there is also a situation where the same conclusion as in Corollary~\ref{cor:compact-operators-oc-norm} holds, 
but where $X$ is not assumed to be reflexive. 
Recall that a Banach space $Y$ is said to have the \emph{approximation property} if the identity operator on $Y$ 
can be approximated, uniformly on compact sets, by compact linear operators. 
Equivalently, for every Banach space $X$ every compact linear operator $X \to Y$ can be approximated in operator norm 
by finite rank operators. 
Hence, if $X'$ and $Y$ are separable (with respect to the norm topology) and $Y$ has the approximation property, 
then the space $\calK(X;Y)$ of compact linear operators is separable with respect to the operator norm.
From this one gets the following result,  
whose assertion and proof were kindly brought to the author's attention by Wolfgang Arendt. 

\begin{corollary}
	\label{cor:compact-operators-oc-norm-non-reflexive} 
	Let $X,Y$ be Banach lattices, where $X'$ and $Y$ are separable and have order continuous norm 
	and where $Y$ has the approximation property. 
	Then in the ordered Banach space $\calK(X;Y)$ of all compact linear operators from $X$ to $Y$ 
	every increasing order bounded sequence is norm convergent,  
	i.e.\ the space $\calK(X;Y)$ satisfies the equivalent conditions of Theorem~\ref{thm:oc-norm}.
\end{corollary}

Interestingly, the proof of Corollary~\ref{cor:compact-operators-oc-norm-non-reflexive} does not rely 
on a version of Dini's theorem but instead on Theorem~\ref{thm:separable}:

\begin{proof}[Proof of Corollary~\ref{cor:compact-operators-oc-norm-non-reflexive}]
	As discussed before the corollary, the separability of $X'$ and $Y$ and the approximation property of $Y$ 
	imply that $\calK(X;Y)$ is separable. 
	So by Theorem~\ref{thm:separable} it suffices to show that every increasing order bounded sequence in $\calK(X;Y)$ has a supremum. 
	Let $(T_n)$ be such a sequence, say in the positive cone, and let $S \in \calK(X;Y)$ be an upper bound of $(T_n)$.
	The ordered boundedness together with the assumption that $Y$ has order continuous norm, 
	implies that $(T_n)$ converges strongly to an operator $T \in \calL(X,Y)$. 
	Clearly, $0 \le T \le S$. 
	As in the previous corollary, the order continuity of the norms of $X'$ and $Y$ implies 
	that compactness of operators is inherited under domination 
	(see e.g.\ \cite[Theorem~16.20]{AliprantisBurkinshaw1985} or \cite[Theorem~3.7.13]{MeyerNieberg1991}), 
	so $T$ is also in $\calK(X;Y)$. 
	The strong convergence of $(T_n)$ to $T$ now gives that $T$ is the supremum of this sequence within $\calK(X;Y)$.
\end{proof}

The assumption that $X'$ be separable and has order continuous norm is, for instance, satisfied 
if $X$ is separable and reflexive -- but for reflexive $X$, Corollary~\ref{cor:compact-operators-oc-norm} 
gives the same conclusion anyway and has weaker additional assumptions, so this case is not interesting.
The space $X = c$ of all convergent sequences is an example 
where Corollary~\ref{cor:compact-operators-oc-norm-non-reflexive} can be applied 
while Corollary~\ref{cor:compact-operators-oc-norm} cannot. 
Another such example is the space $X = c_0$ of all sequences that converge to $0$.

Note that in both Corollaries~\ref{cor:compact-operators-oc-norm} and~\ref{cor:compact-operators-oc-norm-non-reflexive} 
the assumption that $X$ and $Y$ be Banach lattices is only needed to have criteria available which imply 
that compactness of operators is inherited under domination. 
The author is not aware of any such criterion in the more general setting of ordered Banach spaces.

\section{When is the bidual wedge a cone?}
\label{sec:bidual-wedge-pointed}

In the subsequent Section~\ref{sec:normality-dini} we will discuss to which extent normality is needed 
for a Dini type result such as Corollary~\ref{cor:dini-weak-conv} to hold. 
As a preparation we use the present section to discuss for an ordered Banach space $X$ 
under which conditions the bidual wedge $X''_+$ is pointed, i.e.\ in other words, a cone. 
Note that it follows from the Hahn--Banach theorem that $X''_+$ is a cone if and only if the span of $X'_+$ is norm dense 
(equivalently, weakly dense) in $X'$.

The question when $X''_+$ is a cone is related to (non-)normality of $X_+$ as follows: 
if $X_+$ is a normal cone, then the dual wedge $X'_+$ is generating in $X'$ and hence the bidual wedge $X''_+$ 
is again a cone (and in fact even a normal cone). 
Thus, the condition that the bidual wedge $X''_+$ be a cone is weaker 
than the condition that the cone $X_+$ be normal.
By comparing the following two propositions one can see a nice analogy between $X_+$ being normal and $X''_+$ being a cone.
The first proposition is a version of \cite[Theorem~2.23]{AliprantisTourky2007}; 
we include it here to provide context for the second proposition and since it is used in 
Example~\ref{exas:bidual-cone}\ref{exas:bidual-cone:itm:ell-1} below.

\begin{proposition}
	\label{prop:normal-dom}
	For an ordered Banach space $X$ the following are equivalent:
	\begin{enumerate}[label=\upshape(\roman*)]
		\item\label{prop:normal-dom:itm:normal}
		The cone $X_+$ is normal. 
		
		\item\label{prop:normal-dom:itm:dom-bdd} 
		If two norm bounded nets $(x_j)_{j \in J}$ and $(y_j)_{j \in J}$ in $X$ satisfy 
		$0 \le x_j \le y_j$ for all $j$ 
		and if $y_j \to 0$ in norm, then also $x_j \to 0$ in norm. 
		
		\item\label{prop:normal-dom:itm:dom-bdd-sequence}
		If two norm bounded sequences $(x_n)$ and $(y_n)$ in $X$ satisfy 
		$0 \le x_n \le y_n$ for all $n$ 
		and if $y_n \to 0$ in norm, then also $x_n \to 0$ in norm. 
		
		\item\label{prop:normal-dom:itm:dom} 
		If two (not necessarily norm bounded) nets $(x_j)_{j \in J}$ and $(y_j)_{j \in J}$ in $X$ satisfy 
		$0 \le x_j \le y_j$ for all $j$ 
		and if $y_j \to 0$ in norm, then also $x_j \to 0$ in norm. 
	\end{enumerate}
\end{proposition}

\begin{proof}
	``\ref{prop:normal-dom:itm:normal} $\Rightarrow$ 
	\ref{prop:normal-dom:itm:dom} $\Rightarrow$ 
	\ref{prop:normal-dom:itm:dom-bdd} $\Rightarrow$ 
	\ref{prop:normal-dom:itm:dom-bdd-sequence}''
	These implications are obvious.

	\impliesProof{prop:normal-dom:itm:dom-bdd-sequence}{prop:normal-dom:itm:normal} 
	We show the contrapositive, so let $X_+$ be non-normal. 
	Then there exists a vector $y \in X_+$ such that the order interval $[0,y]$ is not norm bounded 
	\cite[Theorem~2.40(1) and~(4)]{AliprantisTourky2007}. 
	Thus, there exists a sequence $(w_n)$ in $[0,y]$ such that $0 < \norm{w_n} \to \infty$. 
	Define $x_n := \frac{w_n}{\norm{w_n}}$ and $y_n := \frac{y}{\norm{w_n}}$ for each $n$. 
	Then $0 \le x_n \le y_n$ for every $n$ and $y_n \to 0$ in norm. 
	The sequence $(x_n)$, on the other hand, is normalized and does thus not converge to $0$ in norm. 
	As both sequences $(x_n)$ and $(y_n)$ are norm bounded, 
	it follows that~\ref{prop:normal-dom:itm:dom-bdd-sequence} fails.
\end{proof}

\begin{proposition}
	\label{prop:bidual-cone-dom}
	For an ordered Banach space $X$ the following are equivalent:
	\begin{enumerate}[label=\upshape(\roman*)]
		\item\label{prop:bidual-cone-dom:itm:cone} 
		The bidual wedge $X''_+$ is pointed, i.e.\ a cone. 
		
		\item\label{prop:bidual-cone-dom:itm:dom} 
		If two norm bounded nets $(x_j)_{j \in J}$ and $(y_j)_{j \in J}$ in $X$ satisfy 
		$0 \le x_j \le y_j$ for all $j$ 
		and if $y_j \to 0$ weakly, then also $x_j \to 0$ weakly.  
	\end{enumerate}
\end{proposition}

\begin{proof}
	\impliesProof{prop:bidual-cone-dom:itm:cone}{prop:bidual-cone-dom:itm:dom}
	Assume that $X''_+$ is a cone 
	and let $(x_j)_{j \in J}$ and $(y_j)_{j \in J}$ be as in~\ref{prop:bidual-cone-dom:itm:dom}. 
	Consider $X$ as a subspace of $X''$. 
	We first note that if $(x_j)_{j \in J}$ is weak* convergent in $X''$, then its limit is $0$, 
	and thus $x_j \to 0$ weakly in $X$ in this case.
	Indeed, if the limit is denoted by $x'' \in X''$, 
	then the weak convergence $y_j \to 0$ and the domination $0 \le x_j \le y_j$ for all $j \in J$ 
	imply that $0 \le x'' \le 0$. 
	Since $X''_+$ is a cone, it follows that $x'' = 0$. 
	
	The previous observation can be applied to every subnet of $(x_j)_{j \in J}$. 
	This together with the Banach--Alaoglu theorem implies that every subnet of $(x_j)_{j \in J}$ 
	has a subnet that weakly converges to $0$ in $X$. 
	Hence, $(x_j)_{j \in J}$ is itself weakly convergent to $0$.
	
	\impliesProof{prop:bidual-cone-dom:itm:dom}{prop:bidual-cone-dom:itm:cone}
	We prove the contrapositive, so assume that $X''_+$ is not a cone. 
	Then there exists a non-zero vector $x'' \in X''_+$, say of norm $\norm{x''} = 1$, such that $0 \le x'' \le 0$. 
	
	According to Theorem~\ref{thm:positive-ball-bidual} in the appendix, 
	the positive closed unit ball $B_+ := \{x \in X_+ \suchthat \norm{x} \le 1\}$ of $X$ is weak* dense 
	in the positive closed unit ball $B''_+ := \{x'' \in X''_+ \suchthat \norm{x''} \le 1\}$ of $X''$. 
	As both $x''$ and $-x''$ are an element of $B''_+$ we can find two nets $(x_j)_{j \in J}$ and $(w_j)_{j \in J}$ 
	in $B_+$ that are weak* convergent to $x''$ and $-x''$, respectively. 
	Note that we can choose both nets to have the same index set: 
	whenever a point $p$ is in the closure of a subset $S$ of a topological vector space, 
	there exists a net in $S$ that converges to $p$ and whose index set is the neighbourhood filter of $0$.
	
	Both nets $(x_j)_{j \in J}$ and $(y_j)_{j \in J} := (x_j + w_j)_{j \in J}$ are norm bounded and located in $X_+$ 
	and they satisfy $x_j \le y_j$ for each $j$. 
	One has $y_j \to x''-x'' = 0$ with respect to the weak* topology in $X''$ and thus, weakly in $X$. 
	However, $x_j \to x'' \not= 0$ with respect to the weak* topology in $X''$, 
	so $(x_j)_{j \in J}$ does not converge weakly to $0$ in $X$.
\end{proof}

Remarkably, the characterizations in Propositions~\ref{prop:normal-dom} and~\ref{prop:bidual-cone-dom} 
do not immediately unveil the fact, mentioned before Proposition~\ref{prop:normal-dom}, 
that normality of the cone $X_+$ implies that the bidual wedge $X''_+$ is a cone.
Let us now discuss for a few examples and example classes whether the bidual wedge is a cone. 

\begin{examples}
	\label{exas:bidual-cone}
	\begin{enumerate}[label=(\alph*)]
		\item\label{exas:bidual-cone:itm:reflexive} 
		If an ordered Banach space $X$ is reflexive, 
		then the bidual wedge $X''_+$ is a cone. 
		In fact, the canonical identification of $X$ with $X''$ is order preserving in both directions, 
		so one has $X_+ = X''_+$ under this identification.

		\item\label{exas:bidual-cone:itm:sobolev-reflexive} 
		Let $k \ge 1$ be an integer and $p \in (1,\infty)$ 
		and endow the Sobolev space $X := W^{k,p}(-1,1)$ with the pointwise almost everywhere order inherited from $L^p(-1,1)$. 
		Then the bidual wedge $X''_+$ is a cone according to~\ref{exas:bidual-cone:itm:reflexive} since $X$ is reflexive.

		\item\label{exas:bidual-cone:itm:C_k} 
		On the other hand, let $k \ge 1$ be an integer and let $X := C^k([-1,1])$ denote the Banach space 
		of $k$-times continuously differentiable functions on $[-1,1]$ 
		with the norm $\norm{f}_{C^k} := \max\{\norm{f^{(j)}}_\infty \suchthat 0 \le j \le k\}$ 
		and the pointwise order. 
		The the bidual wedge $X''_+$ is not a cone. 
		
		To see this, we use the characterization in Proposition~\ref{prop:bidual-cone-dom}. 
		Choose a sequence of functions $(f_n)$ in $C^k([-1,1])$ such that $0 \le f_n \le \frac{1}{n} \one$, 
		$\norm{f_n}_{C^k} \le 1$ and $f_n^{(k)}(0) = 1$ for each $§n \ge 1$. 
		The existence of such a function can, for each $n$, be seen as follows: 
		start with a continuous function that maps into $[0,1]$, has integral $\frac{1}{n}$ and maps $0$ to $1$; 
		then integrate it $k$ times from $-1$ to its argument in order to get $f_n$.
		
		The sequence $(\frac{1}{n}\one)$ converges to $0$ in norm and thus, in particular, weakly. 
		However, the sequence $(f_n)$ does not converges weakly to $0$ 
		since $f_n^{(k)}(0) = 1$ for each $n$ and since $g \mapsto g^{(k)}(0)$ is a continuous linear functional on $C^k([0,1])$. 
		Thus, we can apply Proposition~\ref{prop:bidual-cone-dom}\ref{prop:bidual-cone-dom:itm:dom} 
		to the sequences $(f_n)$ and $(\frac{1}{n}\one)$ to see that $X''_+$ is not a cone.

		\item\label{exas:bidual-cone:itm:sobolev-one} 
		Again, let $k \ge 1$ be an integer and endow the Sobolev space $X := W^{k,1}(-1,1)$ with the pointwise almost everywhere order 
		inherited from $L^1(-1,1)$. 
		Then the bidual wedge $X''_+$ is not a cone. 
		
		To see this, we will again use Proposition~\ref{prop:bidual-cone-dom}. 
		We first define a (non-positive) sequence $(h_n)$ in $L^1$ in the following way. 
		Choose a strictly increasing sequence $(x_n)$ in $[0,1)$ that converges to $1$ 
		and that satisfies $\frac{x_{n+1} - x_n}{1-x_n} \to 0$. 
		Such a sequence $(x_n)$ exists; 
		for instance, it is not difficult to check that the numbers 
		$x_n := \frac{1}{\sum_{k=1}^\infty \frac{1}{k^2}} \sum_{k=1}^n \frac{1}{k^2}$ satisfy the required properties 
		(but it is worthwhile to point out that the partial sums of, say, the geometric series 
		$\sum_{k=1}^\infty \frac{1}{2^k}$ cannot be taken as the $x_n$ since they do not satisfy $\frac{x_{n+1} - x_n}{1-x_n} \to 0$).
		
		For each index $n$ we define $h_n$ to be $0$ on the left of $x_n$; 
		for every $j \ge n$ we set $h_n$ to be equal to $\frac{1}{1-x_n}$ on the left half of the interval $[x_j, x_{j+1}]$ 
		and to be $-\frac{1}{1-x_n}$ on the right half of the same interval.
		Then $\norm{h_n}_{L^1} = 1$ for each $n$ and the sequence $h_n$ does not converge weakly to $0$ in $L^1$
		as can be seen by testing against the indicator function $\one_A \in L^\infty(-1,1)$, 
		where $A \subseteq (-1,1)$ is the union of the left halfs of all the intervals $[x_j, x_{j+1}]$.
		
		Now define $f_n \in W^{k,1}(-1,1)$ for each $n \ge 0$ by integrating the function $h_n$ 
		from $-1$ to its argument $k$ times. 
		Then each $f_n$ is in the positive cone $W^{k,1}(-1,1)_+$ and the sequence $(f_n)$ is norm bounded in 
		$W^{k,1}(-1,1)$. 
		Each $h_n$ vanishes on $(-1,0)$ and one has 
		\begin{align*}
			\int_{-1}^t h_n(s) \dx s 
			\le 
			\frac{1}{2}\frac{x_{n+1} - x_n}{1-x_n}  
		\end{align*}
		for each $n \in \bbN$ and each $t \in [0,1)$. 
		This implies that $f_n \le \frac{1}{2}\frac{x_{n+1} - x_n}{1-x_n} \one_{[0,1]}$ for each $n \in \bbN$. 
		The functions on the right of this inequality converge to $0$ with respect to the norm in $W^{k,1}(-1,1)$ 
		since $\frac{x_{n+1} - x_n}{1-x_n} \to 0$ by the choice of the points $x_n$.
		But the sequence $(f_n)$ itself does not converge weakly to $0$ in $W^{k,1}(-1,1)$ 
		since $(h_n)$ does not converges weakly to $0$ in $L^1(-1,1)$. 
		So Proposition~\ref{prop:bidual-cone-dom} shows that the bidual wedge of $W^{k,1}(-1,1)_+$ is not a cone. 
		
		\item\label{exas:bidual-cone:itm:non-reflexive-cone} 
		Here is an example of a non-reflexive space with a non-normal cone for which the bidual wedge is a cone: 
		Let $X = c_0$ be the space of real sequences that converge to $0$, say with index set $\bbN_0 := \{0,1,2,\dots\}$, 
		and endow this space with the sup norm and the cone 
		\begin{align*}
			(c_0)_+ := \Big\{x \in c_0 \suchthat x_0 \ge \sup_{k \ge 1} \frac{\modulus{x_k}}{k} \Big\}.
		\end{align*}
		It is easy to check that $(c_0)_+$ is indeed a cone and is closed. 
		The cone is not normal since $e_0 \ge n e_n$ for each $n \ge 0$, where the $e_n$ denote the canonical unit vectors. 
		The point $e_0$ is an interior point of $(c_0)_+$, 
		so $(c_0)_+$ has non-empty interior and is thus, in particular, generating.
		
		To see that the bidual wedge in $\ell^\infty$ is a cone, 
		we use the characterization from Proposition~\ref{prop:bidual-cone-dom}. 
		Let $(x_j)_{j \in J}$ and $(y_j)_{j \in J}$ be norm bounded nets in $c_0$ 
		such that $0 \le x_j \le y_j$ for all $j \in J$ 
		and assume that $(y_j)_{j \in J}$ converges weakly to $0$. 
		For each $j \in J$ it follows from $0 \le x_j \le y_j$ 
		that $0 \le (x_j)_0 \le (y_j)_0$ and 
		that $\frac{1}{k} \modulus{(y_j)_k - (x_j)_k} \le (y_j)_0 - (x_j)_0$ for all $k \ge 1$. 
		Thus, $(x_j)_{j \in J}$ converges to $0$ componentwise 
		and hence, due to the norm boundedess, also weakly. 
		So Proposition~\ref{prop:bidual-cone-dom} implies that the bidual wedge of $(c_0)_+$ is indeed a cone. 
		
		\item\label{exas:bidual-cone:itm:ell-1} 
		Let $X$ be $\ell^1$ or, more generally, any Banach space with the \emph{Schur property} 
		which means that every weakly convergent sequence in $X$ is automatically norm convergent. 
		If $X$ is endowed with a non-normal closed cone $X_+$, 
		then the bidual wedge $X''_+$ is never a cone. 
		(A concrete example of a non-normal cone in $\ell^1$ can be found in Example~\ref{exa:ell-1-non-normal} below.)
		
		To see this, assume that $X''_+$ is a cone. 
		We show that this implies normality of $X_+$ 
		by checking that condition~\ref{prop:normal-dom:itm:dom-bdd-sequence} 
		in Proposition~\ref{prop:normal-dom} is satisfied. 
		So let $(x_n)$ and $(y_n)$ be norm bounded sequences in $X_+$ such that $x_n \le y_n$ for all $n$ 
		and assume that $y_n \to 0$ in norm. 
		Then, in particular, $y_n \to 0$ weakly and since $X''_+$ is a cone, 
		Proposition~\ref{prop:bidual-cone-dom} implies that $x_n \to 0$ weakly, too. 
		As $X$ has the Schur property it follows that even $x_n \to 0$ in norm, 
		so Proposition~\ref{prop:normal-dom} implies that $X_+$ is normal.
	\end{enumerate}
\end{examples}

The main reason why we discussed the condition that $X''_+$ be a cone in this section 
is that this condition occurs in Theorem~\ref{thm:no-dini} in the next section. 
Before we proceed to this, we end the present section with a brief digression 
that shows another nice consequence if $X''_0$ is a cone. 

A positive element $u$ of an ordered Banach space $X$ is called an \emph{almost interior point} 
if $\langle x', u \rangle > 0$ for every non-zero functional $0 \le x' \in X'$. 
A detailed discussion of almost interior point along with several examples can be found in \cite[Section~2]{GlueckWeber2020}. 

\begin{proposition}
	\label{prop:almost-interior}
	Let $X, Y$ be ordered Banach spaces and let $u \in X_+$ be an almost interior point. 
	\begin{enumerate}[label=\upshape(\alph*)]
		\item\label{prop:almost-interior:itm:functionals} 
		If $(x'_j)$ is a norm bounded net in $X'_+$ such that $\langle x'_j, u \rangle \to 0$, 
		then $(x'_j)$ is weak* convergent to $0$.
		
		\item\label{prop:almost-interior:itm:operators} 
		Assume that the bidual wedge $Y''_+$ is a cone. 
		If $(T_j)$ is a norm bounded net of positive linear operators from $X$ to $Y$ and
		$T_j u \to 0$ weakly, then $T_j \to 0$ with respect to the weak operator topology.
	\end{enumerate}
\end{proposition}

\begin{proof}
	\ref{prop:almost-interior:itm:functionals} 
	First note that if $(x'_j)$ is weak* convergent, then its limit has to be $0$. 
	Indeed, if $x'$ denotes the limit, then $x' \ge 0$ and $\langle x', u \rangle = 0$, 
	so $x' = 0$ since $u$ is an almost interior point.
	
	By applying this observation to subnets of $(x'_j)$ and using the Banach--Alaoglu theorem 
	we see that every subnet of $(x'_j)$ has a subnet that converges weak* to $0$, 
	so $(x'_j)$ is itself weak* convergent to $0$.
	
	\ref{prop:almost-interior:itm:operators} 
	Fix $x \in X$. 
	For every $y' \in Y'_+$ one has $\langle T_j' y', u \rangle \to 0$ by assumption, 
	so it follows from~\ref{prop:almost-interior:itm:functionals} that the norm bounded net $(T_j'y')$ 
	converges weak* to $0$ in $X'$. 
	Thus, $\langle T_j x, y' \rangle \to 0$ for each $y'$ in the span of $Y'_+$. 
	Since $Y''_+$ is a cone in $Y''$ by assumption, the span of $Y'_+$ is norm dense in $Y'$ 
	and hence we conclude that the norm bounded net $(T_j x)$ in $Y$ converges weakly to $0$.
\end{proof}

\section{Normality of the cone in Dini's theorem}
\label{sec:normality-dini}

It is not difficult to see that the assumptions of Theorem~\ref{thm:dini} imply that the cone $X_+$ is normal. 
Indeed, this is a consequence of the existence of a norm bounded set $C' \subseteq X'_+$ that satisfies property~(3) in the theorem.
It is natural to ask whether there exist examples of ordered Banach spaces whose cone is not normal, 
but in which a version of Dini's theorem still holds. 

Towards this end, the following observation was pointed out in \cite{WongNg1975}: 
the space $\ell^1$ is well-known to have the \emph{Schur property}, 
i.e., every weakly convergent sequence is automatically norm convergent. 
Thus, if one endows $\ell^1$ with any closed non-normal cone, then one has constructed in ordered Banach space 
with non-normal cone in which every increasing weakly convergent sequence is norm convergent. 
For the sake of completeness, let us show by a concrete example that there exists a non-normal cone in $\ell^1$.

\begin{example}
	\label{exa:ell-1-non-normal} 
	Let the underlying index set of the real Banach space $\ell^1$ be $\bbN_0 := \{0,1,2,\dots\}$.
	The set 
	\begin{align*}
		\ell^1_+ := \Big\{x \in \ell^1 \suchthat x_0 \ge \sum_{k=1}^\infty \frac{\modulus{x_k}}{2^k} \Big\}
	\end{align*}
	is a closed cone in $\ell^1$ which is not normal, but which has non-empty interior and is thus generating.
\end{example}

\begin{proof}
	It is straightforward to check that the set $\ell^1_+$ is a closed cone. 
	The $0$-th canonical unit vector $e_0$ is an interior point of $\ell^1_+$. 
	Hence, $\ell^1_+$ has non-empty interior and is thus, in particular, generating. 
	
	To see that $\ell^1_+$ is not normal, let $e_k \in \ell^1$ denote the $k$-th canonical unit vector for each $k \in \bbN_0$. 
	For every $k \ge 1$ one then has $e_0 \ge 2^k e_k$ with respect to the order induced by $\ell^1_+$, 
	which shows the non-normality.
\end{proof}

Note that the Schur property of $\ell^1$ is a condition for sequences only 
-- there exist weakly convergent nets in $\ell^1$ that are not norm convergent. 
It is not clear -- at least not to the author -- 
whether in the space $\ell^1$ endowed with the cone from Example~\ref{exa:ell-1-non-normal} 
every increasing weakly convergent net is norm convergent.
This raises the question whether the preceding example is only an artefact 
of the fact that one considers sequences instead of nets or whether, on the other hand, 
the validity of Dini's theorems for all nets implies that the cone is normal. 
Let us pose this as an open problem:

\begin{open_problem}
	\label{op:dini-normal}
	Let $X$ be an ordered Banach space and assume that every weakly convergent increasing net 
	is norm convergent. 
	Does it follow that $X_+$ is normal?
\end{open_problem}

Let us also note that, since a weakly convergent net need not be eventually norm bounded, 
the answer to Open Problem~\ref{op:dini-normal} might change if one only considers weakly convergent increasing nets 
that are, in addition, norm bounded. 

Finally, we give a partial answer to Open Problem~\ref{op:dini-normal} by showing that, 
in a large class of spaces with non-normal cone, 
Dini's theorem does not hold -- not even for sequences.

\begin{theorem}
	\label{thm:no-dini}
	Let $X$ be an ordered Banach space with non-normal cone. 
	Assume that the space $X$ is reflexive or that, more generally, the bidual wedge $X''_+$ is a cone.
	Then there exists an increasing sequence in $X_+$ that converges weakly but not in norm.
\end{theorem}

Spaces where the assumption that $X''_+$ be a cone is satisfied 
can, for instance, be found in Examples~\ref{exas:bidual-cone}\ref{exas:bidual-cone:itm:sobolev-reflexive} 
and~\ref{exas:bidual-cone:itm:non-reflexive-cone}. 
Note that, according to Example~\ref{exas:bidual-cone}\ref{exas:bidual-cone:itm:ell-1} 
that assumption that $X_+$ be non-normal and $X''_+$ be a cone 
is never satisfied if $X$ has the Schur property. 
This is consistent with the discussion before Example~\ref{exa:ell-1-non-normal}.

For the proof of Theorem~\ref{thm:no-dini} we need the following lemma.

\begin{lemma}
	\label{lem:non-normal}
	Let $X$ be an ordered Banach space and let $x \in X_+$. 
	If the order interval $[0,x]$ is not norm bounded, then there exists an increasing sequence in $[0,x]$ 
	that is not norm bounded.
\end{lemma}

\begin{proof}
	For each integer $n \ge 1$ the order interval $[0,\frac{x}{2^n}]$ is not norm bounded, 
	so we can find a vector $x_n \in [0,\frac{x}{2^n}]$ of norm $\norm{x_n} = 3^n$. 
	Define $y_n := \sum_{k=1}^n x_k$ for each $n \ge 1$. 
	The the sequence $(y_n)$ is increasing and contained in $[0,x]$. 
	Moreover, for each $n \ge 1$ the geometric sum formula gives $\norm{y_n} \le \frac{3^{n+1}}{2} = \frac{\norm{x_{n+1}}}{2}$
	and thus $\norm{y_{n+1}} \ge \norm{x_{n+1}}-\norm{y_n} \ge \frac{\norm{x_{n+1}}}{2} \to \infty$ as $n \to \infty$.
\end{proof}

\begin{proof}[Proof of Theorem~\ref{thm:no-dini}]
	It suffices to show the existence of a decreasing sequence in $X_+$ which converges weakly but not in norm. 
	This is technically a bit easier to write down.
	
	As $X_+$ is not normal, there exists an $x \in X_+$ such that the order interval $[0,x]$ is not norm bounded 
	\cite[Theorem~2.40(1) and~(4) on p.\,87]{AliprantisTourky2007}. 
	According to Lemma~\ref{lem:non-normal} we find an increasing sequence in $[0,x]$ 
	that is not norm bounded and 
	by substracting this sequence from $x$ we get a decreasing sequence in $[0,x]$ that is not norm bounded. 
	By then choosing an appropriate subsequence we obtain a decreasing sequence $(x_n)$ in $[0,x]$ 
	such that $\big(\norm{x_n}\big)$ is increasing and converges to $\infty$. 
	It follows that the sequence $(y_n)$ in $X_+$ defined by $y_n := \frac{x_n}{\norm{x_n}}$ for each $n$ 
	is also decreasing. 
	Since all the $y_n$ are normalized, the sequence $(y_n)$ is not norm convergent to $0$, 
	so to conclude the proof it suffices to show that it converges weakly to $0$. 
	
	To this end, observe that $0 \le y_n \le \frac{x}{\norm{x_n}}$ for every $n$ 
	and that the right hand side of this inequality converges to $0$ in norm and thus, in particular, weakly. 
	Since the bidual wedge $X''_+$ is assumed to be a cone, 
	it follows from the characterization of that property in Proposition~\ref{prop:bidual-cone-dom} 
	that $y_n \to 0$ weakly.
\end{proof}

In addition to Open Problem~\ref{op:dini-normal} the discussion in this section 
leaves the following problem open, which is a bit more vague than the previous problems. 

\begin{open_problem}
	\label{op:dini-sequences}
	Characterize those ordered Banach spaces $X$ (with non-normal cone) 
	in which every increasing weakly convergent sequence is norm convergent.
\end{open_problem}

One such characterization was actually given in \cite[Theorem~2]{WongNg1975}: 
the property that every increasing weakly convergent sequence is norm convergent 
is equivalent to the formally weaker property 
that every increasing and weakly convergent sequence with norm convergent subsequence is itself norm convergent. 
While this is a remarkable observation, it seems unclear how this can be applied to concrete examples 
of ordered Banach spaces with non-normal cone. 
Thus, we consider Problem~\ref{op:dini-sequences} as not fully solved, yet; 
further characterizations (or criteria) that can be used to analyse concrete spaces still seem to be missing.
The negative criterion in Theorem~\ref{thm:no-dini} is one step in this direction.

\subsection*{Acknowledgements} 

It is my pleasure to thank Wolfgang Arendt for raising the question whether the assertion of Theorem~\ref{thm:separable} holds, 
for an inspiring discussion about the meaning of the notion \emph{order continuous norm} in ordered Banach spaces 
which motivated Theorem~\ref{thm:oc-norm}, 
and for bringing to my attention that Theorem~\ref{thm:separable} 
yields Corollary~\ref{cor:compact-operators-oc-norm-non-reflexive} as a consequence; 
this also inspired Corollary~\ref{cor:compact-operators-oc-norm}. 
Moreover, I would like to thank Julian Hölz for a fruitful discussion about the appropriate assumptions 
for a general version of Dini's theorem, which inspired Theorem~\ref{thm:dini}.

Some of the materials in Section~\ref{sec:dini} were contained in the lecture notes 
for a course on \emph{Ordered Banach Spaces and Positive Operators} that I taught at the University of Wuppertal 
in summer 2023.

\appendix

\section{Weak* density of the positive unit ball}
\label{app:weak-star-dense}

The following theorem is used in the proof of Proposition~\ref{prop:bidual-cone-dom} 
and it is also interesting in its own right. 
Since the proof uses quite a bit technology about polars in dual pairs of vector spaces 
that does not occur in the rest of the paper, we outsource the result to this appendix.

As a motivation first note that, if $X$ is a pre-ordered Banach space, then the wedge $X_+$ is weak* dense in the bidual wedge $X''_+$. 
This follows from the Hahn--Banach extension theorem in the locally convex space $X''$ endowed with the weak* topology. 
The following theorem shows that an even stronger approximation result is true. 
It can be considered as an order version of Goldstine's theorem 
(and it actually contains Goldstine's theorem as the special case where the wedge $X_+$ is equal to $X$).

\begin{theorem}
	\label{thm:positive-ball-bidual}
	Let $X$ be a pre-ordered Banach space and consider the positive parts of the closed unit balls 
	in $X$ and $X''$, given by  
	\begin{align*}
		B_+ := \{x \in X_+ \suchthat \norm{x} \le 1\} 
		\quad \text{and} \quad 
		B''_+ := \{x'' \in X''_+ \suchthat \norm{x''} \le 1 \}
		.
	\end{align*}
	Then $B_+$ is weak* dense in $B''_+$.
\end{theorem}

For the proof we use two ingredients. 
The first one is the following lemma. 

\begin{lemma}
	\label{lem:ball-plus-cone-dual}
	Let $X$ be a pre-ordered Banach space and let $B'$ denote the closed unit ball of the dual space $X'$. 
	\begin{enumerate}[label=\upshape(\alph*)]
		\item\label{lem:ball-plus-cone-dual:itm:closed} 
		The set $B' + X'_+$ in $X'$ is weak* closed.
		
		\item\label{lem:ball-plus-cone-dual:itm:closure} 
		In the space $X'$ the weak closure and the weak* closure of the convex hull $\conv(B' \cup X'_+)$ 
		coincide and are equal to $B' + X'_+$.
	\end{enumerate}
\end{lemma}

\begin{proof}
	\ref{lem:ball-plus-cone-dual:itm:closed} 
	Let $(b'_j+x'_j)_{j \in J}$ be a net in $B'+X'_+$ that is weak* convergent to a point $y' \in X'$. 
	After replacing the net with a subnet we may, by the Banach--Alaoglu theorem, 
	assume that $(b'_j)_{j \in J}$ is weak* convergent to a point $b' \in B'$. 
	Hence, $x'_j \to y'-b' =: x'$ with respect to the weak* topology 
	and $x' \in X'_+$ since $X'_+$ is weak* closed.
	Thus, $y' = b'+x' \in B' + X'_+$.
	
	\ref{lem:ball-plus-cone-dual:itm:closure}
	The set $B'+X'_+$ is convex as the Minkowski sum of two convex sets
	and it contains $B' \cup X'_+$, so $\conv(B' \cup X'_+) \subseteq B'+X'_+$. 
	Thus, 
	\begin{align*}
		\overline{\conv(B' \cup X'_+)}^{\operatorname{w}} 
		\subseteq 
		\overline{\conv(B' \cup X'_+)}^{\operatorname{w}^*} 
		\subseteq 
		B'+X'_+
		,
	\end{align*}
	where the closures are taken with respect to the weak and the weak* topology, respectively, 
	and where the second inclusion holds since $B'+X'$ is weak* closed 
	according to~\ref{lem:ball-plus-cone-dual:itm:closed}.
	
	It only remains to show that $B'+X'_+ \subseteq \overline{\conv(B' \cup X'_+)}^{\operatorname{w}}$, 
	so consider a vector $b'+x' \in B'+X'_+$. 
	For every number $\delta \in (0,1)$ the vector 
	\begin{align*}
		\delta b' + x' = \delta b' + (1-\delta) \frac{x}{1-\delta}
	\end{align*}
	is located in $\conv(B' \cup X'_+)$, 
	so by taking the limit for $\delta \uparrow 1$
	we see that $b' + x'$ is in the norm closure and thus in the weak closure of this set.
\end{proof}

The second ingredient that we need for the proof of Theorem~\ref{thm:positive-ball-bidual} 
is the concept of \emph{polars} in dual pairs of vector spaces. 
Let $V$, $W$ be real vector spaces and let $\langle \argument, \argument \rangle: V \times W \to \bbR$ 
be a bilinear map such that the following holds for all $v_0 \in V$ and all $w_0 \in W$: 
the equality $\langle v_0, w \rangle = 0$ for all $w \in W$ implies $v_0 = 0$ 
and the equality $\langle v, w_0 \rangle = 0$ for all $v \in V$ implies $w_0 = 0$. 
We shall describe this situation by speaking of the \emph{dual pair} $\langle V, W\rangle$. 
For a dual pair $\langle V,W \rangle$, the space $V$ induces a Hausdorff locally convex topology on $W$ 
and vice versa.
In the proof of Theorem~\ref{thm:positive-ball-bidual} we will be interesting 
in the dual pairs $\langle X, X' \rangle$ and $\langle X'', X' \rangle$ for a Banach space $X$.

Let $\langle V, W\rangle$ be a dual pair. 
For every subset $C \subseteq V$ the \emph{polar} $C^\circ$ of $C$ is defined to be the subset
\begin{align*}
	C^\circ := \{w \in W \suchthat \langle c,w \rangle \le 1 \text{ for all } c \in C\}
\end{align*}
of $W$ 
and for every subset $D \subseteq W$ the \emph{polar} $D^\circ$ of $D$ is defined to be the subset
\begin{align*}
	D^\circ := \{v \in v \suchthat \langle v,d \rangle \le 1 \text{ for all } d \in D\}
\end{align*}
of $V$. 
Naturally, we can thus also define the \emph{bipolars} $C^{\circ\circ} := (C^\circ)^\circ \subseteq V$ 
and $D^{\circ\circ} := (D^\circ)^\circ \subseteq W$. 
Note that in the definition of polars we did not use absolute values around the terms $\langle c,w \rangle$ and $\langle v,d \rangle$. 
There is also a symmetric version of polars that uses absolute values; 
we follow the convention from \cite[pp.\,125--126]{SchaeferWolff1999}, though, 
which is better suited for working with cones and wedges.
We will need the following properties, 
where we endow $V$ with the topology induced by $W$ and vice versa.

\begin{enumerate}[label=(\alph*)]
	\item\label{polars:itm:union} 
	For two subsets $D_1, D_2 \subseteq W$ one has $(D_1 \cup D_2)^\circ = D_1^\circ \cap D_2^\circ$; 
	see \cite[property~3 in paragraph~IV.1.3 on the bottom of p.\,125]{SchaeferWolff1999}. 
	
	\item\label{polars:itm:intersection} 
	If $C_1, C_2 \subseteq V$ are convex and closed and both sets contain $0$, 
	then $(C_1 \cap C_2)^\circ$ is the closure of the convex hull of $C_1^\circ \cup C_2^\circ$ in $W$; 
	see \cite[Corollary~2 in paragraph~IV.1.5 on p.\,126]{SchaeferWolff1999}.
	
	\item\label{polars:itm:bipolar} 
	If $C \subseteq V$ is convex and contains $0$, 
	then the bipolar $C^{\circ \circ}$ is the closure of $C$; 
	this is a special case of the bipolar theorem \cite[paragraph~IV.1.5 on p.\,126]{SchaeferWolff1999}. 
	
	\item\label{polars:itm:closed-convex-hull} 
	For every $D \subseteq W$ the polar $D^\circ$ coincides with the polar of the closure of the convex hull of $D$; 
	indeed, this is a simple consequence of the definition of the polar.
\end{enumerate}

\begin{proof}[Proof of Theorem~\ref{thm:positive-ball-bidual}]
	Throughout the proof we will identify $X$ with a subspace of $X''$ in the canonical way 
	and we will use following simple but important observation:
	for every set $C \subseteq X$ the polar of $C$ 
	with respect to the dual pair $\langle X,X'\rangle$ 
	is a subset of $X'$ that coincides with the polar of $C$ 
	with respect to the dual pair $\langle X'',X' \rangle$. 
	Hence, we can unambiguously denote this polar by $C^\circ$.
	Let $B$, $B'$, and $B''$ denote the closed unit balls in $X$, $X'$, and $X''$, respectively.
	We proceed in several steps:
	
	\emph{Step~1:} 
	We observe that $B_+^\circ$ is the weak* closure of $\conv(B' \cup -X'_+)$ in $X'$. 
	
	Indeed, the set $B_+$ is the intersection of the weakly closed subsets $B$ and $X_+$ of $X$, 
	so applying property~\ref{polars:itm:intersection} listed before the proof 
	to the dual pair $\langle X, X' \rangle$ shows that
	$B_+^\circ$ is the weak* closure of $\conv(B^\circ \cup X_+^\circ)$ in $X'$. 
	It is not difficult to check that $B^\circ = B'$ and $X_+^\circ = -X'_+$, so the claim follows.
	
	\emph{Step~2:} 
	We show that, with respect to the dual pair $\langle X'', X'\rangle$, 
	one has $B_+^{\circ\circ} = (B' \cup -X'_+)^\circ$. 
	
	Since it does not make a difference to take the polar $B_+^\circ$ with respect to the dual pair $\langle X, X' \rangle$ 
	instead of $\langle X'', X' \rangle$
	(see the beginning of the proof), 
	it follows from Step~1 that $B_+^\circ$ is the weak* closure of $\conv(B' \cup -X'_+)$ in $X'$. 
	By applying Lemma~\ref{lem:ball-plus-cone-dual}\ref{lem:ball-plus-cone-dual:itm:closure} 
	to the space $X$ ordered by the cone $-X_+$ we see that this weak* closure actually coincides with the weak closure, 
	so $B_+^\circ = \overline{\conv(B' \cup -X'_+)}^{\operatorname{w}}$ in $X'$. 
	So by applying property~\ref{polars:itm:closed-convex-hull} listed before the proof 
	to the dual pair $\langle X'', X' \rangle$ we get the claimed equality.
	
	\emph{Step~3:} 
	We prove that $B_+$ is weak* dense in $B''_+$, which completes the proof. 
	
	To this end note that we have, with respect to the dual pair $\langle X'', X'\rangle$,
	\begin{align*}
		B_+^{\circ\circ} 
		= 
		(B' \cup -X'_+)^\circ 
		= 
		(B')^\circ \cap (-X'_+)^\circ
		,
	\end{align*}
	where the first equality was shown in Step~2 
	and the second equality follows from property~\ref{polars:itm:union} listed before the proof. 
	It is straighforward to check that $(B')^\circ = B''$ and $(-X'_+)^\circ = X''_+$, 
	so 
	\begin{align*}
		B_+^{\circ\circ} = B'' \cap X''_+ = B''_+.	
	\end{align*}
	The bipolar theorem in the dual pair $\langle X'', X' \rangle$ 
	shows that $(B_+)^{\circ\circ}$ is the weak* closure of $B_+$, 
	see property~\ref{polars:itm:bipolar} listed before the proof. 
	So $B''_+$ is the weak* closure of $B_+$, as claimed.
\end{proof}

\bibliographystyle{plain}
\bibliography{literature}

\end{document}